\documentclass[11pt,british,reqno]{amsart}

\usepackage[T1]{fontenc}

\usepackage{babel,eucal,url,amssymb,stmaryrd,booktabs,enumerate,amscd,paralist,cite,color,comment}

\theoremstyle{plain}
\newtheorem{lemma}{Lemma}[section]
\newtheorem{definition}[lemma]{Definition}
\newtheorem{proposition}[lemma]{Proposition}
\newtheorem{corollary}[lemma]{Corollary}

\newtheorem{remark}[lemma]{Remark}
\newtheorem{example}[lemma]{Example}

\DeclareMathOperator{\Ric}{Ric}

\newcommand{\Lie}[1]{\operatorname{\textsl{#1}}}
\newcommand{\lie}[1]{\operatorname{\mathfrak{#1}}}

\newcommand{\un}{\lie{u}}
\newcommand{\Gtwo}{\ifmmode{{\rm G}_2}\else{${\rm G}_2$}\fi}

\newcommand{\Nt}{\nabla^{\Lie{U}(n)}}

\newcommand{\Div}{\textup{div}}

 \newcommand{\cyclic}{\mathop{\kern0.9ex{{+}\kern-2.2ex\raise-.28ex\hbox{\Large\hbox
 {$\circlearrowright$}}}}}

\newcommand{\real}[1]{\left\llbracket #1 \right\rrbracket}

\newcommand{\inp}[2]{\langle #1, #2\rangle}

\newcommand{\alt}{\mathbf a}
\newcommand{\lcf}{\lbrack\!\lbrack}
\newcommand{\rcf}{\rbrack\!\rbrack}

\newcommand{\Ricac}{\Ric^{*}} 

\def\sideremark#1{\ifvmode\leavevmode\fi\vadjust{\vbox to0pt{\vss
 \hbox to 0pt{\hskip\hsize\hskip1em
 \vbox{\hsize2.5cm\tiny\raggedright\pretolerance10000
 \noindent #1\hfill}\hss}\vbox to8pt{\vfil}\vss}}}%

\newfont{\eusm}{eusm10 scaled \magstep1}
\newfont{\eusmiii}{eusm10 scaled \magstep3}

\title{The exterior derivative of the Lee form of  almost Hermitian manifolds}

\author{Francisco~Mart\'\i n~Cabrera}
\address[Francisco~Mart\'\i n~Cabrera]{Departamento de Matemáticas, Estadística e Investigación  Ope\-ra\-ti\-va \\
  University of La Laguna\\ 38200 La Laguna, Tenerife, Spain}
\email{fmartin@ull.edu.es}

\newcommand{\m}{P}
\renewcommand\le{\leqslant}

\newcommand{\ZZ}[1]{\al#1\ar^\perp\kern-8pt\lower2pt\hbox{$_\cZ$}}
\newcommand{\EE}[1]{\lceil#1\kern-0.5pt\rfloor}

\newcommand{\cZ}{\mathcal{Z}}

\newcommand{\dm}{d\kern1pt\raise.8pt\hbox{$^\m$}}

\newcommand{\E}{\raise1pt\hbox{$\textstyle\bigwedge$}\kern-0.5pt}
\newcommand{\ba}{\begin{array}}\newcommand{\ea}{\end{array}}
\newcommand{\be}{\begin{equation}}
\newcommand{\ee}[1]{\label{#1}\end{equation}}
\newcommand{\bt}{\begin{tabular}}\newcommand{\et}{\end{tabular}}
\newcommand{\al}{\langle}\newcommand{\ar}{\rangle}

\begin{document}

\maketitle

\begin{center}
{\footnotesize To the memory of Thomas Friedrich}
\end{center}

\begin{abstract}
  The exterior derivative $d \theta$ of the Lee form $\theta$ of almost Hermitian manifolds is studied. If $\omega$ is the Kähler two-form, it is proved that  the $\mathbb{R}\omega$-component of  $d\theta$ is always zero. Expressions    for the  other components, in $[\lambda_0^{1,1}]$ and in $\lcf \lambda^{2,0}\rcf$,  of $d\theta$   are also obtained.  They are given  in terms  of the intrinsic torsion. Likewise,  
  it is   described some   interrelations between the Lee form and $\Lie{U}(n)$-components of the 
  Riemannian 
  curvature tensor. 
\end{abstract}
\vspace{3mm}

{\footnotesize


}
\vspace{4mm}

\section{Introduction} 

In \cite{Gray-H:16} Gray and Hervella  displayed a  classification for almost Hermitian structures. Such a classification is based on the  decomposition of the space   possible intrinsic torsions into irreducible $\Lie{U}(n)$-modules. Since they obtained   a decomposition into four irreducible modules $\mathcal{W}_i$ for $n>2$,  $i=1,\dots ,4$,   there are  $2^{4}$ classes of almost structure denoted by  direct sums of  $\mathcal{W}_i$ determined for the non-zero components $\xi_{(i)}$ of the intrinsic torsion $\xi$. The component $\xi_{(4)}$ is  determined by a one-form $\theta$, usually called the \textit{Lee form} in references \cite{Lee}.

The exterior derivative $d\theta$ of the Lee form $\theta$ amounts interest because such a form is related with conformal changes of the almost Hermitian metric. For instance, if $\theta$ is closed,  at least locally, it is possible to do  a conformal change of metric such that  $\xi_{(4)}=0$ for the new almost Hermitian structure. In the particular case that $\xi_{(1)}=0$, $\xi_{(3)}=0$ and $n>2$, i.e. the almost Hermitian structure is of type $\mathcal{W}_2 \oplus \mathcal{W}_4$, it is not hard to deduce that the Lee form $\theta$ must be  closed.  From this, a natural question arises: are there another types of almost Hermitian structures with $\theta \neq 0$ such that $\theta$ is necessarily closed for them?. The interest of this question is increased by the fact that the most  examples in references with $\theta \neq  0$  are such that $\theta$ is closed. 
On the other hand, if the exterior derivative $d \theta$ of $\theta$ is non-zero, under action of $\Lie{U}(n)$, $d \theta$ is decomposed into $\Lie{U}(n)$-components, when do  some of them vanish?.

The initial purpose of the present text is to obtain answers for the mentioned questions. Thus  we will give expressions for the   $\Lie{U}(n)$-components, $(d \theta)_{\mathbb{R} \omega}$, $(d \theta)_{[\lambda^{1,1}_0]}$ and $(d \theta)_{\lcf \lambda^{2,0} \rcf}$,   of the exterior derivative  $d \theta$  in terms of the intrinsic torsion.  It is relevant that  it is  given a  proof for the fact  that the component  $(d \theta)_{\mathbb{R} \omega}$, proportional to the Kähler two-form $\omega$, always vanishes (see Proposition \ref{divergenciaunouno}).
The identity     $(d \theta)_{\mathbb{R} \omega} =0$ has already  been obtained by Gauduchon in \cite{Gau}. He gave  another  proof in the context  of Hermitian structures (type $\mathcal{W}_3 \oplus \mathcal{W}_4$) which is also valid for general almost Hermitian structures.

 In order to obtain  expressions for $(d \theta)_{\mathbb{R} \omega}$, $(d \theta)_{[\lambda^{1,1}_0]}$ and $(d \theta)_{\lcf \lambda^{2,0} \rcf}$,  we will make use of some interrelations among components of the intrinsic torsion which are consequences of the identity  $d^2 \omega =0$. 
 Such interrelations are interesting by their own. 
 For instance, one of them  has been used in \cite{Martin3} to explain  the behavior of certain components of the Riemannian curvature.  Likewise some of these interrelations have been applied   in the study of harmonic $\Lie{U}(n)$-structures  (see \cite{GDMC}). 
Finally, we point out that  the expressions obtained for   $(d \theta)_{[\lambda^{1,1}_0]}$ and $(d \theta)_{\lcf \lambda^{2,0} \rcf}$ (see Proposition \ref{divergenciaunouno})   allow us to say when some of them are zero. For instance, if the almost Hermitian manifold is of type $\mathcal{W}_1 \oplus \mathcal{W}_4$, then $(d \theta)_{[\lambda^{1,1}_0]}$ vanishes. Moreover, if the manifold is of dimension six, then $\theta$ is closed for such a type 
      (see Proposition \ref{w1w4}).

The Riemannian curvature tensor of almost Hermitian manifolds has been studied by 
 Tricerri an Vanhecke  \cite{Tricerri-Vanhecke:aH}, Falcitelli, Farinola and Salamon  
\cite{Falcitelli-FS:aH} and by  Swann and the present author   \cite{Martin3}. Section 4 could  be considered as an application of  the expressions for the  $\Lie{U}(n)$-components   of $d \theta$ to the study of such a curvature tensor. Thus those expressions  provide a better understanding of facts and properties   shown in the mentioned references. For instance, it is explained the behavior of certain $\Lie{U}(n)$-components of the Riemannian curvature tensor determined  by means of  the usual Ricci tensor $\Ric$ and  another Ricci type  tensor $\Ricac$ given in \cite{Martin3}. Likewise, some results given in  \cite{Falcitelli-FS:aH} are completed.

In the last section, some examples are studied to show that the two
components of $d \theta$ orthogonal to $\omega$,  $(d \theta)_{[\lambda^{1,1}_0]}$ and $(d \theta)_{\lcf \lambda^{2,0} \rcf}$,  can be   non-zero. The first and second examples are  four-dimensional manifolds  equipped with a Hermitian structure for which
both of these components are non-zero. The third example, which is also
Hermitian and has previously been studied in some details be Abbena et al.
\cite{AbbGarSal}, again has both  components  non-zero.
In these examples  computations relative to the Riemannian curvature are also  done. This illustrates the study of facts and properties displayed in Section 4. 

Finally, we   point out that it is still  an open question to find an example 
 of  almost Hermitian manifold of type $\mathcal{W}_1 \oplus \mathcal{W}_4$ 
such that $d \theta \neq 0$, i.e.   $(d \theta)_{\lcf \lambda^{2,0} \rcf} \neq 0$. Note that such an example ought to   be of  dimension greater than six.

\section{Preliminaries}{\indent}
\label{sect:prelima} \setcounter{equation}{0}
An \textit{ almost complex} structure  on a manifold $M$ consists of a $(1,1)$-tensor $J$ such that
$
J^{2} = -I. 
$
The manifold $M$ must be of dimension $2n$. The presence of an almost complex structure is equivalent to say that there is a $\Lie{GL}(n, \mathbb{C}) $-structure defined  on $M$. 
 A manifold $M$ is said to be   \textit{almost Hermitian}, if there is an almost complex structure and a   Riemannian metric 
$\langle \cdot, \cdot \rangle$  defined on $M$ such that they  satisfy   the compatibility condition 
$\langle J
X, JY\rangle = \langle X,Y\rangle$. 
In this case it is said that there is a $\Lie{U}(n)$-structure  on $M$.

Associated with an  almost Hermitian  structure,   the tensor $\omega = \langle \cdot, J
\cdot\rangle$, called the {\em Kähler form}, is usually
considered. Using $\omega$, $M$ can be oriented by fixing a
constant multiple of $\omega^n$ as volume form. Under the action of 
$\Lie{U}(n)$, the cotangent space on each point
$\mbox{T}^{*}_m M$ is  irreducible  and  it follows 
$
\lie{so}(2n) \cong \Lambda^{2} \mbox{T}^* M = \lie{u}(n) \oplus \lie{u}(n)^{\perp} $, where
$\lie{u}(n)$ ($\lie{u}(n)^{\perp} $)
consists of those two-forms $b$ such that $b(J X , J Y)
= b( X  ,  Y )$ ($b(JX , JY) = - b( X , Y)$). 
 
Denoting by $\nabla$ the Levi Civita connection,  the {\it minimal
connection} $\nabla^{\Lie{U}(n)}$  is  the unique $\Lie{U}(n)$-connection on $M$ such that  $\xi = \nabla^{\Lie{U}(n)} -\nabla$  satisfies the condition $\xi \in \mbox{T}^* M \otimes  \lie{u}(n)^{\perp}$. The tensor  $\xi$ is referred to as the {\it intrinsic
torsion} of the almost Hermitian  structure 
\cite{CleytonSwann:torsion}.  In \cite{Gray-H:16}, Gray and Hervella showed that in general dimensions  the space $\mbox{T}^* M \otimes \lie{u}(n)^{\perp}$
of possible intrinsic torsions    is
decomposed into four  irreducible $\Lie{U}(n)$-modules 
providing a classification of $2^4=16$ classes or types of almost Hermitian structures. To be more precise,  the space of possible intrinsic torsions
  $\mbox{T}^* M\otimes  \lie{u}(n)^{\perp}$ consists of those tensors
$\xi$ such that
$  J \xi_X Y + \xi_X J Y = 0$
and, under the action of $U(n)$, is decomposed into:
\begin{enumerate}
\item if $n=2$, $
\xi \in \mbox{T}^* M \otimes \un(2)^\perp = \mathcal
W_{2} \oplus \mathcal W_{4}$;
\item if $n \geqslant 3$, $ \xi \in \mbox{T}^* M
\otimes \un(n)^\perp =
  \mathcal W_{1}  \oplus   \mathcal W_{2}  \oplus  \mathcal W_{3}  \oplus  \mathcal W_{4}$.
\end{enumerate}
The intrinsic  torsion is explicitly determined   by
 $ \xi_X   =   -  \tfrac12  J \circ \nabla_X J
 =  \tfrac12 (\nabla_X J) \circ J. $ 
 The tensor $\xi_{(i)}$ will denote the  component of 
$\xi$ corresponding to $\mathcal{W}_i$. 
The component $\xi_{(4)}$  is determined   by    the vector field $\sum_{i=1}^{2n} \xi_{e_i} 
 e_i$, where $\{e_1, \dots , e_{2n} \}$ is a local  orthonormal frame field. Other expression for this vector field is 
  $
 2  \textstyle   \sum_{i=1}^{2n}   \xi_{e_i} e_i  =  - J (d^* \omega )^{\sharp},
$
where $d^*$ denotes coderivative and $\langle  (d^* \omega )^{\sharp}, X \rangle = d^*\omega(X)$. Since the \textit{Lee form}, considered in \cite{Lee, Gray-H:16}, is defined by $\theta = -\frac{1}{n-1}  J d^* \omega $, one has 
$
\textstyle \sum_{i=1}^{2n}   (\xi_{e_i} e_i)^\flat = \tfrac{n-1}{2} \theta,  
$
where $(\xi_{e_i} e_i)^\flat(X) = \langle \xi_{e_i} e_i , X \rangle$.  We point out  that, for a one-form $a$, $(J a)(X) := - a (JX)$. The component $\xi_{(4)}$ is explicitly given by  
$$
4 \xi_{(4)X} = X^\flat \otimes \theta^\sharp - \theta \otimes X - J X^\flat \otimes J\theta^\sharp + J \theta \otimes JX.  
$$
\begin{remark}[Notations and conventions]
  {\rm For using simpler and standard notation,  we recall that $\lambda_0^{p,q}$ is a complex irreducible $\Lie{U}(n)$-module coming from the $(p,q)$-part of the  
complex exterior algebra, and that its corresponding dominant weight in standard
coordinates is given by $(1, \dots,1,0, \dots,0, -1, \dots , -1)$, where $1$ and $-1$ are repeated $p$ and $q$ times, respectively. By analogy with the exterior algebra, there
are also complex irreducible $\Lie{U}(n)$-modules $\sigma^{p,q}_0$, with dominant weights $(p,0, \dots ,0, -q)$ 
coming from the complex symmetric algebra. The notation $\lcf V \rcf$ stands for the real vector space underlying a complex vector space $V$, and $[W ]$ denotes a real vector space that admits $W$ as its complexification. Thus for the $\Lie{U}(n)$-modules above mentioned one has
\begin{gather*}
 \mbox{T}^* M  \cong \lcf \lambda^{1,0} \rcf \cong \mathcal W_4, \quad  \lie{u}(n)  \cong  [\lambda^{1,1}], \quad  \lie{su}(n) \cong    [\lambda^{1,1}_0] ,  \quad \lie{u}(n)^{\perp} \cong  \lcf \lambda^{2,0} \rcf, 
 \\
\quad \mathcal W_1   \cong \lcf \lambda^{3,0} \rcf, \quad  \mathcal W_2   \cong \lcf A \rcf, \qquad \mathcal W_3   \cong \lcf \lambda^{2,1}_0 \rcf,
\end{gather*}
where $A \subset \lambda^{1,0} \otimes \lambda^{2,0}$.

We will use the natural extension  to forms of the metric $\langle \cdot , \cdot \rangle$. Thus, for all $p$-forms $\alpha$, $\beta$,  
 $$
 \langle \alpha , \beta \rangle = \tfrac{1}{p!} \textstyle  \sum_{i_1, \dots, i_p=1}^{2n} \alpha (e_{i_1} , \dots , e_{i_p})
\beta (e_{i_1} , \dots , e_{i_p}).
$$ 
For instance, using this product, for a two-form $\alpha$,  one has  $\alpha_{\mathbb{R}\omega} = \frac1{n} \langle \alpha , \omega \rangle \omega$. Another example using this product is the identity $J d^* \omega = - \langle \cdot \lrcorner d \omega, \omega \rangle$, where $\lrcorner$ denotes the interior product.
 
  In the sequel, we will consider the orthonormal basis for tangent vectors   $\{e_1, \dots, e_{2n}\}$. Likewise, we will use the summation convention to simplify notation. The repeated indexes will mean that the sum is extended from $i=1$ to $i=2n$.  Otherwise, the sum will be explicitly  written. 
 We also point out that  we will make reiterated use of the \textit{musical isomorphisms} $\flat  :  \mathrm{T} M \to  \mathrm{T}^* M$
 and $\sharp : \mathrm{T}^* M \to \mathrm{T} M$, induced by  $\langle \cdot , \cdot \rangle$, defined  by $X^\flat = \langle X, \cdot  \rangle$  and $\langle a^\sharp , \cdot \rangle  = a$.  Finally,  if $\psi$ is a $(0,s)$-tensor,  we write
\begin{gather*} \label{conv:A}
  J_{(i)}\psi(X_1, \dots, X_i, \dots , X_s) = - \psi(X_1, \dots , JX_i, \dots ,
  X_s).
\end{gather*}
   }
  \end{remark}

\section{The components of the exterior derivative of the Lee form}
In this section we will display  several identities relating components of the intrinsic torsion which are consequences of the equalities $d^2 \omega=0$. They are  interesting  by their own and  we will show some applications of them.
   For instance, from such   identities we will obtain expressions for  the $\Lie{U}(n)$-components of the exterior derivative of the Lee form.   
 \begin{lemma} \label{lambdaunouno}
  For an almost Hermitian  manifold of dimension $2n$, $n>1$, the following identities are  satisfied:
 \begin{align*}
   0 = & \; \langle \nabla^{\Lie{U}(n)}_{e_j}  \xi_{(4)e_i} e_i , J e_j \rangle,\\
   0   = & 
   \;
   - \tfrac{n-2}{n-1} \textstyle   \langle \nabla^{\Lie{U}(n)}_{  X} \xi_{(4)e_i} { e_i},    Y \rangle
   +  \tfrac{n-2}{n-1} \textstyle   \langle \nabla^{\Lie{U}(n)}_{ Y} \xi_{(4) e_i} { e_i},    X \rangle
   - 2 \textstyle  \langle (\nabla^{\Lie{U}(n)}_{e_i} \xi_{(3)})_{ X} Y,  { e_i} \rangle
    \\
    &
    + 2 \textstyle  \langle (\nabla^{\Lie{U}(n)}_{e_i} \xi_{(3)})_{ Y}   X, { e_i} \rangle
  - \tfrac{n-2}{n-1} \textstyle  \langle \nabla^{\Lie{U}(n)}_{J  X} \xi_{(4)e_i} { e_i}, J  Y \rangle
   +  \tfrac{n-2}{n-1} \textstyle   \langle \nabla^{\Lie{U}(n)}_{J Y} \xi_{(4)e_i} { e_i}, J X \rangle
\\
&
- 3 \textstyle  \langle \xi_{(1)X} e_i, \xi_{(2)Y} e_i \rangle 
+ 3  \textstyle  \langle \xi_{(1)Y} e_i, \xi_{(2)X} e_i \rangle,   
\\  
     0 = &
     \;
      3 \textstyle  \langle (\nabla^{\Lie{U}(n)}_{e_i} \xi_{(1)} )_{e_i} X_{}, Y_{} \rangle
   - \textstyle  \langle (\nabla^{\Lie{U}(n)}_{e_i} \xi_{(3)} )_{e_i} X_{}, Y_{} \rangle
   + (n-2) \textstyle   \langle (\nabla^{\Lie{U}(n)}_{e_i} \xi_{(4)} )_{e_i} X_{},Y_{}
   \rangle
  \\
 &
  - \textstyle   \langle \xi_{{(3)}X_{}} e_i,  \xi_{{(1)} Y_{}} e_i \rangle
  + \textstyle   \langle  \xi_{{(3)}Y_{}} e_i,  \xi_{{(1)}  X_{} }e_i 
  +  \tfrac12\textstyle   \langle  \xi_{{(3)}X_{}} e_i,  \xi_{{(2)} Y_{}} e_i \rangle
   - \textstyle   \frac12  \langle  \xi_{{(3)}Y_{}} e_i,  \xi_{{(2)}  X_{}} e_i \rangle
      \\
 &
 \textstyle   - \tfrac{n-5}{n-1}  \textstyle  \langle \xi_{{(1)} \xi_{{(4)}e_i} e_i }X_{}, Y_{} \rangle
    - \tfrac{n-2}{n-1}  \textstyle  \langle  \xi_{{(2)} \xi_{{(4)}e_i} e_i } X_{}, Y_{} \rangle
      +  \textstyle   \langle \xi_{{(3)} \xi_{(4)e_i} e_i } X_{}, Y_{} \rangle.
      \end{align*}
 \end{lemma}
 \begin{remark}
 {\rm The  third identity of Lemma \ref{lambdaunouno} has  already been shown in \cite{Martin3}. There it was used to explain some aspects of the behavior of the Riemannian curvature tensor.}
 \end{remark}
\begin{proof}
Consider the K\"ahler form~$\omega$.  Being a differential form it
satisfies $d^2\omega=0$.  However, since the Levi-Civita
connection~$\nabla$ is torsion-free, we may compute $d^2\omega$ using
$\nabla$.  Writing $\nabla=\nabla^{\Lie{U}(n)}-\xi$ and using $\nabla^{\Lie{U}(n)}\omega=0$, we have first
that
\begin{equation*}
  \tfrac12d\omega(Y,Z,W) = \langle \xi_YZ, JW\rangle  + \langle \xi_WY, JZ \rangle  +
  \langle \xi_ZW, JY \rangle.
\end{equation*}
Now $d^2\omega = \alt(\nabla^{\Lie{U}(n)}\omega)- \alt(\xi d\omega)$, where $\alt\colon
T^*M\otimes \Lambda^3T^*M\to\Lambda^4T^*M$~is the alternation map.  One
computes that these two terms are the expressions obtained respectively by
summing $\varepsilon \inp{(\Nt_X\xi)_YZ}{JW}$ and $\varepsilon
\inp{\xi_{\xi_XY}Z}{JW}$ over all permutations of $(X,Y,Z,W)$, where
$\varepsilon$~is the sign of the permutation. After doing all of this we obtain
\begin{align}
 d^2 \omega (X_1,X_2,X_3,X_4)  = & \textstyle \sum_{1 \leq a < b \leq 4} (-1)^{a+b} \left( \left((\nabla^{\Lie{U}(n)}_{X_a} \xi)_{X_b} - (\nabla^{\Lie{U}(n)}_{X_b} \xi)_{X_a}\right) \omega \right) (X_c, X_d) \nonumber
   \\
& 
\textstyle + \sum_{1 \leq a < b \leq 4} (-1)^{a+b}(\xi_{\xi_{X_a} X_b - \xi_{X_b} X_a}  \omega) (X_c, X_d) \label{d2F}
\\
&
\textstyle - \sum_{1 \leq a < b \leq 4} (-1)^{a+b}([\xi_{X_a} ,\xi_{X_b}]  \omega) (X_c, X_d), \nonumber
  \end{align}
 where $c < d$, $ \{c,d\} =\{1, \dots,4\}-\{a,b\}$ in each case and $[\xi_{X_a} , \xi_{X_b} ] = \xi_{X_a}  \xi_{X_b} - \xi_{X_b}  \xi_{X_a}$.

We have that
\begin{equation*}
  \Lambda^4T^*M =
  \lcf \lambda^{4,0} \rcf +\lcf \lambda^{3,1} \rcf+\lcf \lambda^{2,0} \rcf \omega +
  [\lambda^{2,2}_0]+[\lambda^{1,1}_0]\omega+\mathbb R\omega^2,
\end{equation*}
so in order to compute the  components in $\left[\lambda^{1,1}\right] = \mathbb{R} + [\lambda^{1,1}_0 ]$ and $\real{\lambda^{2,0}}$   of $d^2\omega$, 
we contract with $\omega$ on the first two arguments. Then we take the
corresponding projections to $\left[\lambda^{1,1}\right]$ and  $\lcf \lambda^{2,0} \rcf$, which are respectively the   $1$-eigenspace and  $(-1)$-eigenspace of $J$
acting on $2$-forms.   Using the symmetries of the components of~$\xi$,  one
obtains  the components in $[\lambda^{1,1}]$ and   $\lcf \lambda^{2,0}\rcf$ of $d^2\omega$  written in terms of $\nabla^{\Lie{U}(n)}$ and $\xi$. Such components   vanish because $d^2\omega=0$.  
For the  first identity, we  do a contraction with $\omega$ on the  component in $[\lambda^{1,1} ]$. In this way we will obtain the equality 
$$
 0=  8 \langle \nabla^{\Lie{U}(n)}_{e_j}  \xi_{(4)e_i} e_i , J e_j \rangle - 3 (n-1)\langle \xi_{(1)e_j} e_i , \xi_{(2)e_j} J e_i \rangle. 
$$
Since $\xi_{(2)} \circ J$ is still in $\mathcal{W}_2$, $\xi_{(2)} \circ J$ is orthogonal to  $\xi_{(1)}$. Hence $\langle \xi_{(1)e_j} e_i , \xi_{(2)e_j} J e_i \rangle=0$ and the first identity follows.
Taking this into account  in  the $[\lambda^{1,1}]$-component, we will obtain the second identity. 
Finally,  the third identity follows by considering   the component in $\lcf \lambda^{2,0}\rcf$. 
  \end{proof}
  \begin{remark}{\rm 
  It is well known that the respective  curvature tensors $R$ and $R^{\Lie{U}(n)}$ of the connections $\nabla$ and $\nabla^{\Lie{U}(n)}$ are related by 
  $$
R{(X,Y)} = R^{\Lie{U}(n)}{(X,Y)} + (\nabla^{\Lie{U}(n)}_{X} \xi)_Y - (\nabla^{\Lie{U}(n)}_{Y}
\xi)_X + \xi_{ \xi_X Y}  - \xi_{\xi_YX} - [ \xi_X ,
\xi_Y ]
$$
(see \cite{CleytonSwann:torsion}). Using this identity in \eqref{d2F}, it is obtained
$$
0 =\sum_{1\leq a < b \leq 4} (-1)^{a+b}   \left(  \left( R(X_a , X_b) - R^{\Lie{U}(n)}(X_a, X_b) \right) \omega \right) ( X_c, X_{d}). 
$$
where $c<d$, $\{c,d\}=\{1, \dots ,4\} - \{a,b\}$ in each summand. We stress that $\nabla^{\Lie{U}(n)}\omega=0$ is a key fact here.  }
  \end{remark}
\vspace{2mm}

Next we  note that 
\begin{gather*}
 \textstyle    (\nabla^{\Lie{U}(n)}_X (\xi_{(4)e_i} { e_i})^\flat) (Y) - (\nabla^{\Lie{U}(n)}_Y (\xi_{(4)e_i} { e_i})^\flat) (X)  = 
\textstyle  d  (\xi_{(4)e_i} { e_i})^\flat)(X,Y) -  \langle \xi_X Y - \xi_Y X , \xi_{(4)e_i} { e_i} \rangle, \nonumber\\
\textstyle (n-1)   \langle (\nabla^{\Lie{U}(n)}_{e_i} \xi_{(4)} )_{e_i} X_{},Y_{} \rangle  
  =
  \textstyle ( d (\xi_{(4)e_i} e_i)^\flat)_{\lcf \lambda^{2,0} \rcf} (X,Y) - 2  \langle \xi_{(1) \xi_{(4)e_i} e_i} X, Y \rangle 
  +   \textstyle  \langle \xi_{(2) \xi_{(4)e_i} e_i} X, Y \rangle \nonumber  
  \end{gather*}
(for an explicit proof for this second identity, see Lemma 4.4 in \cite{GDMC}). Using these identities in previous Lemma, we will obtain the components of the exterior derivative of the Lee form $\theta$. 
 \begin{proposition} \label{divergenciaunouno}
 For almost Hermitian manifolds of dimension $2n$, $n>1$, the following identities are satisfied:
 \begin{align*}
   \textstyle   (d \theta)_{\mathbb{R} \omega}  = & \;0,
   \\
    \tfrac{n-2}{2} \textstyle  (d \theta )_{[\lambda^{1,1}_0]} ( X,  Y)  
=&
   - \textstyle  \langle (\nabla^{\Lie{U}(n)}_{e_i} \xi_{(3)})_{ X} Y,  { e_i} \rangle
    + \textstyle   \langle (\nabla^{\Lie{U}(n)}_{e_i} \xi_{(3)})_{ Y}   X, { e_i} \rangle
  +  \tfrac{n-2}{2} \textstyle  \langle \xi_{(3)  X} Y- \xi_{(3)  Y} X, \theta^\sharp \rangle
  \\
  & 
  - \tfrac{3}{2} \textstyle  \langle \xi_{(1)X} e_i, \xi_{(2)Y} e_i \rangle 
  + \tfrac{3}{2} \textstyle  \langle \xi_{(1)Y} e_i, \xi_{(2)X} e_i \rangle,  
  \\   
    \tfrac{n-2}{2} \textstyle  (d \theta )_{\lcf \lambda^{2,0} \rcf} (X,Y)  
    = &
   -   3 \textstyle  \langle (\nabla^{\Lie{U}(n)}_{e_i} \xi_{(1)} )_{e_i} X_{}, Y_{} \rangle
   + \textstyle  \langle (\nabla^{\Lie{U}(n)}_{e_i} \xi_{(3)} )_{e_i} X_{}, Y_{} \rangle
  + \textstyle   \langle \xi_{{(3)}X_{}} e_i,  \xi_{{(1)} Y_{}} e_i \rangle
 \\
 &
  - \textstyle   \langle  \xi_{{(3)}Y_{}} e_i,  \xi_{{(1)}  X_{} }e_i \rangle
  -  \tfrac12\textstyle   \langle  \xi_{{(3)}X_{}} e_i,  \xi_{{(2)} Y_{}} e_i \rangle
   + \textstyle  \tfrac12    \langle  \xi_{{(3)}Y_{}} e_i,  \xi_{{(2)}  X_{}} e_i \rangle
    \\
   &
 \textstyle   + \tfrac{3 (n-3)}{2}  \textstyle  \langle \xi_{{(1)} \theta^\sharp }X_{}, Y_{} \rangle
      - \tfrac{n-1}{2} \textstyle   \langle \xi_{{(3)} \theta^\sharp } X_{}, Y_{} \rangle. 
  \end{align*}
 \end{proposition}
 \begin{remark}
{\rm  The identity     $(d \theta)_{\mathbb{R} \omega} =0$ has already  been obtained by Gauduchon in \cite{Gau} in the context of Hermitian structures (type $\mathcal{W}_3 \oplus \mathcal{W}_4$). However, his proof is valid in the general context of almost Hermitian structures. }
 \end{remark}
 As a consequence of the expressions for the components of $d\theta$ given in Proposition \ref{divergenciaunouno}, we have the following proposition whose part (i) is already well known.
 \begin{proposition} \label{w1w4}
For almost Hermitian manifolds of dimension $2n$ with $n>2$,  we have:
 \begin{enumerate}
\item[$\mathrm{(i)}$] If the structure is of type $ \mathcal{W}_2 \oplus \mathcal{W}_4 \cong \real{ A}\oplus \real{ \lambda^{1,0}} $, then the Lee form $\theta$ is closed.
\item[$\mathrm{(ii)}$] If the structure is of type $\mathcal{W}_1 \oplus \mathcal{W}_4 \cong \lcf  \lambda^{3,0} \rcf  \oplus  \lcf \lambda^{1,0}\rcf $,  then  $(d \theta )_{[\lambda^{1,1}_0]}$  vanishes. In particular, if $n=3$, $d\theta=0$ for such a type. Moreover, if $\xi_{(1)} =0$ on some point, then $\xi_{(1)} =0$ on the  whole corresponding connected component, i.e. the structure is of the type $\mathcal{W}_4$ called locally conformal Kähler structure.  
If $\xi_{(1)} \neq 0$, the Lee form is given by $\theta = d \ln \frac1{\| \xi_{(1)} \|^2}$ on the connected component. Thus, in this second case, the structure of type   $\mathcal{W}_1 \oplus \mathcal{W}_4$  is globally  conformal to the type $\mathcal{W}_1$ and it makes sense to say that we have a globally conformal  nearly Kähler structure.  
 \end{enumerate}
 \end{proposition}
 \begin{proof}
 As we have already said,  (i) and the main part of (ii) are easily deduced from the expressions for  $(d \theta )_{[\lambda^{1,1}_0]}$ and $ (d \theta )_{\lcf \lambda^{2,0} \rcf}$ given in Proposition \ref{divergenciaunouno}. It remains to verify the assertion for $n=3$ in $(ii)$. If the component $d\omega_{\lcf \lambda^{3,0} \rcf}$  in $\lcf \lambda^{3,0}\rcf$ of $d\omega$ is zero,  then $d \theta = 0$ because of (i). Hence in the sequel  we assume  $d\omega_{\lcf \lambda^{3,0} \rcf} \neq 0$ and denote $w_1^+ = \frac16 \| d\omega_{\lcf \lambda^{3,0} \rcf}\|$. Now,  we fix as a complex volume form,  $\Psi = \psi_+ + i \psi_-$, where $\psi_+ = \frac1{3w_1^+} d\omega_{\lcf \lambda^{3,0} \rcf}$ and $\psi_- = J_{(1)} \psi_+$. Then we are in the presence, at a least locally, of a $\Lie{SU}(3)$-structure of  type $\mathcal{W}_1^+ \oplus \mathcal{W}_4 \oplus \mathcal{W}_5$ (see \cite{Martin2} for details). The component of the intrinsic $\Lie{SU}(3)$-torsion in $\mathcal{W}_1^+$ is determined by the component of $d\omega$ in $\mathbb{R} \psi_+ \subseteq \real{\lambda^{3,0}}=\mathcal{W}_1$. On the other hand, the component of the intrinsic $\Lie{SU}(3)$-torsion in $\mathcal{W}_5$ is determined by the one-form $\eta$ which is computed by the identity 
 $$
 \ast ( \ast d \psi_+ \wedge \psi_+ +  \ast d \psi_- \wedge \psi_-) = 4(3 \eta -  \theta), 
 $$
 where $\ast$ is the Hodge star operator with respect to the real volume form $\mathrm{Vol} = -\frac14 \psi_+ \wedge \psi_- = \frac16 \omega \wedge \omega \wedge \omega$. In our situation we have the following exterior derivatives (see \cite{Martin2})
 \begin{align*}
 d \omega   =  3 w_1^+ \psi_+ + \theta \wedge \omega,\quad
 d \psi_+  =  (-3 \eta + \theta ) \wedge \psi_+, \quad
 d \psi_-  =  2  w_1^+ \omega \wedge \omega + (-3 \eta + \theta ) \wedge \psi_-.  
 \end{align*}
 Doing again exterior differentiation, we have 
 $
 0 = d^2 \psi_+ = (-3 d  \eta +  d \theta ) \wedge \psi_+. 
 $
  Hence $d\theta = (d \theta )_{\lcf \lambda^{2,0} \rcf} = 3 (d \eta )_{\lcf \lambda^{2,0} \rcf}$. On the other hand, we obtain 
 \begin{equation} \label{lnwuno}
 0 = d^2 \psi_-  = 2  (  dw_1^+ + w_1^+ (3 \eta  +   \theta )) \wedge \omega \wedge \omega,
 \end{equation}
where we have used  $\psi_+ \wedge \omega =0$ and $(-3 d  \eta +  d \theta ) \wedge \psi_- = (-3 d  \eta +  d \theta )_{\lcf \lambda^{2,0} \rcf}  \wedge \psi_ - =0$. From equation \eqref{lnwuno}, 
$
d (\ln w_1^+) = - 3 \eta - \theta 
$
and $d \theta = - 3  d \eta =  - 3 (d \eta )_{\lcf \lambda^{2,0} \rcf}$. This implies $d \theta=0$.

Now, using the expression for $(d \theta )_{\lcf \lambda^{2,0}\rcf}$ in Proposition \ref{divergenciaunouno} , we have $\langle (\nabla^{\Lie{U}(3)}_{e_i} \xi_{(1)} )_{e_i} X_{}, Y_{} \rangle = 0$ and, in terms of $\Lie{SU}(3)$-structure, here $\xi_{(1)} = \xi_{(1)}^+ = \frac12 w_1^+ \psi_-$. Then 
$$
\langle (\nabla^{\Lie{U}(3)}_{e_i} \xi_{(1)} )_{e_i} X_{}, Y_{} \rangle = \tfrac12 \psi_-((dw_1^+)^\sharp, X,Y) + \tfrac12 w_1^+ (\nabla^{\Lie{U}(3)}_{e_i} \psi_-)(e_i,X,Y).
$$
We recall $ \nabla^{\Lie{SU}(3)} = \nabla^{\Lie{U}(3)}  + \eta = \nabla + \xi + \eta$, where  $\eta_X Y = (J\eta)(X) JY$, and  $\nabla^{\Lie{SU}(3)} \psi_- =0$. Therefore, $\nabla^{\Lie{U}(3)}_{X} \psi_- = - \eta_X \psi_- = 3 (J\eta)(X) \psi_+$. This implies
\begin{eqnarray*}
0 & = &  \langle (\nabla^{\Lie{U}(3)}_{e_i} \xi_{(1)} )_{e_i} X_{}, Y_{} \rangle = \tfrac12 \psi_-((dw_1^+)^\sharp, X,Y) + \tfrac32 w_1^+  \psi_+(J\eta^\sharp,X,Y) 
\\
 &  = & 
  \tfrac12  \psi_-((dw_1^+)^\sharp - 3 w_1^+ \eta^\sharp,X,Y).
\end{eqnarray*}
Hence $d(\ln w_1^+) = 3 \eta$ and, using $d (\ln w_1^+) = - 3 \eta - \theta$, we have $\theta = - 2 d\ln w_1^+ = d\ln \frac{1}{w_1^{+2} }$. It is straightforward to check $w_1^{+2} = \| \xi_{(1)}^+ \|^2$ and, in this situation,  $ \| \xi_{(1)}^+ \|^2 =  \| \xi_{(1)} \|^2$. 

Finally, we will prove that if $\xi_{(1)}=0$ for some point $P$, then $ w_1^+ = \|\xi_{(1)}\|^2 =0$ and $d w_1^+ =0$ at $P$. Hence the function $w_1^+$ is constant and equal to zero on the whole connected component. In fact, if $d w_1^+ \neq 0$ on $P$, there exist a sequence $\{P_i\}_{i\in \mathbb{N}}$ of points converging to $P$ such that $w_1^+(P_i) \neq 0$. Then we have $(d w_1^+)_{P_i} = - \frac12 w_1^+(P_i) \theta_{P_i}$ which converges to $(d w_1^+)_P = - \frac12 w_1^+ (P) \theta_P =0$, contradiction.        
 \end{proof}

\section{Lee form and Riemannian curvature} 
Next we will display relations between the Lee form and $\Lie{U}(n)$-components of the Riemannian curvature  $R$. Some of these components are determined by means of the Ricci tensor $\Ric$ and  a Ricci type tensor $\Ric^*$  associated to the almost Hermitian structure. $\Ricac$ is called the \textit{$\ast$-Ricci curvature tensor} and defined by $\Ricac(X,Y) = \langle R_{X,e_i} JY, Je_i \rangle$, where $R_{X,Y} = \nabla_{[X,Y]} - [\nabla_X, \nabla_Y]$.   Because $\Ricac(JX,JY) = \Ricac(Y,X)$, its symmetric part is in $[\lambda^{1,1}]$ and its skew-symmetric part is in $\lcf \lambda^{2,0}\rcf$. 

\subsection{Some components in the orthogonal complement of the space of Kähler curva\-tu\-res} 
\setcounter{equation}{0}
The components of $R$ determined by the difference $\Ric - \Ric^*$ are  in  $\mathcal{K}_{-1} \cong \mathbb{R}$, $\mathcal{K}_{-2} \cong [\lambda^{1,1}_0]$, $\mathcal{C}_6 \cong \real{\lambda^{2,0}}$ and  $\mathcal{C}_8 \cong \real{\sigma^{2,0}}$ (these notations have been fixed in \cite{Falcitelli-FS:aH}). Such components are included in the orthogonal complement $\mathcal{K}^\perp$ of the space $\mathcal{K}$ of those curvature tensors satisfying the same properties as the curvature of a Kähler manifold, i.e. $\langle R_{X,Y}Z,U\rangle = \langle R_{X,Y}JZ,JU\rangle$.   Such an orthogonal complement is obtained in the space $\mathcal{R}$ of possible Riemannian curvature tensors, i.e. $\mathcal{R} = \mathcal{K} \oplus \mathcal{K}^\perp$.

Because they are included in $\mathcal{K}^\perp$, the above mentioned components of $R$ can be given in terms of the intrinsic torsion. Thus such   an expression for $\Ric - \Ric^*$(see \cite{Martin3}) is 
 \begin{equation} \label{difric}
    \begin{split}
      \Ric (X,Y) - \Ric^* (X,Y)
        & = -2 \inp{(\Nt_{e_i}\xi)_X Y}{e_i} + 2 \inp{(\Nt_X\xi)_{e_i} Y}{e_i} \\
      &\qquad - 2 \inp{\xi_{\xi_{e_i}X}Y}{e_i} + 2 \inp{\xi_{\xi_X e_i}Y}{e_i}.
    \end{split}
  \end{equation}
From this identity,  taking into account  properties of $\xi_{(i)}$ and the fact that $\Nt_X \theta$ is the  Lee form of $\Nt_X \xi$, it is long but straightforward to derive 
  \begin{equation} \label{ricminusricacuno}
    \begin{split}
     ( \Ric -  \Ric^*)_{[\lambda^{1,1}]} (X,Y)
      &  =
          - 2 \inp{(\Nt_{e_i}\xi_{(3)})_X Y}{e_i}
          - \tfrac{n-2}{2} ( ( \nabla_X \theta ) ( Y)  + ( \nabla_{JX} \theta ) ( JY))
  \\
  &\quad \,     
   + \tfrac12 \langle X, Y \rangle (d^* \theta + \tfrac{2n-3}2 \|\theta \|^2 )
   +  4 \inp{\xi_{(1) X} e_i}{\xi_{(1)Y}e_i} 
\\
&
\quad \,  
 - 2 \langle \xi_{(2) e_i} X , \xi_{(2) e_i } Y \rangle 
  - \tfrac{n-2}{4} (\theta(X) \theta(Y) + \theta(JX) \theta(JY))
  \\
  &
 \quad \, - 2\inp{\xi_{(1) X} e_i}{\xi_{(2)Y}e_i}  +   \inp{\xi_{(1)Y}  e_i}{\xi_{(2)X}e_i} 
+(n-2) \theta (\xi_{(3)X} Y).
                  \end{split}
  \end{equation}

Since   $( \Ric -  \Ric^*)_{[\lambda^{1,1}]}$ is symmetric, the skew symmetric part of the right side in \eqref{ricminusricacuno} must be zero. In fact, this is the case, because such a skew-symmetric part is given by one half of  the expression in the right side of the second identity of Lemma  \ref{lambdaunouno} which is  consequence of $d^2 \omega=0$. In other words, the above mentioned skew symmetric part is equal to  $  \tfrac{n-2}{2} (d \theta)_{[\lambda^{1,1}_0]} - A$, where $A$ is the corresponding expression given in Proposition \ref{divergenciaunouno}.  

The component of the curvature in $\mathcal{K}_{-1}$ is  determined by the difference of the scalar curvatures $s$ and $s^*$ with respect to $\Ric$ and $\Ric^*$. By using the previous identity for $\Ric- \Ric^*$, it is obtained the following result.
\begin{lemma} For an almost Hermitian manifold,  we have
\begin{equation} \label{difsca}
s- s^* =  2 (n-1) d^*\theta + (n-1)^2 \|\theta\|^2 +  4 \| \xi_{(1)} \|^2 - 2 \| \xi_{(2)} \|^2, 
\end{equation}
where $\| \xi_{(a)} \|^2 = \langle \xi_{(a)e_i} e_j, \xi_{(a)e_i} e_j\rangle$, $a=1,2$. 
\end{lemma} 
 Next we point out some immediate consequences of \eqref{difsca}.
\begin{proposition} $\,$
\begin{enumerate}
\item[$(i)$] An almost Kähler manifold $(\xi \in \mathcal{W}_2)$  such that $R \in \mathcal{K}$ is Kähler. 
\item[$(ii)$] An almost Hermitian manifold such that  $\xi \in \mathcal{W}_1\oplus  \mathcal{W}_4$,  $d^*\theta \geq 0$ $(\Div \, \theta^\sharp \leq 0)$ and $R \in \mathcal{K}$ is Kähler.
\item[$(iii)$] A compact  almost Hermitian manifold such that  $\xi \in \mathcal{W}_1 \oplus \mathcal{W}_4$ and $R \in \mathcal{K}$ is Kähler.
\end{enumerate}
\end{proposition}
  Claim (i) and a part of  (ii) have been  proved by Falcitelli et al. in  \cite[Prop. 5.5]{Falcitelli-FS:aH}. Other  parts of (ii) and (iii) have been  shown by Vaisman in \cite[Theorem 2.1. (i)]{Vais}.
 \,
  \vspace{2mm}
 
Next we will deduce an expression for the $\lcf \lambda^{2,0}\rcf$-component $\Ric^*_{\lcf \lambda^{2,0}\rcf} = - (\Ric - \Ricac)_{\lcf \lambda^{2,0}\rcf}$ of $\Ric^*$. From \eqref{difric},  taking into account  properties of $\xi_{(i)}$ and the fact that $\Nt_X \theta$ is the  Lee form of $\Nt_X \xi$,  it is obtained 
  \begin{equation} \label{Ricastuno}
    \begin{split}
          \Ric^*_{\real{\lambda^{2,0}}} (X,Y) 
     & 
     =
   \;   2 \inp{(\Nt_{e_i}\xi_{(1)})_{e_i} X}{Y} -   \inp{(\Nt_{e_i}\xi_{(2)})_{e_i} X}{Y}+ \tfrac{n-1}2 (d \theta)_{\lcf \lambda^{2,0} \rcf} (X,Y)
          \\
      & 
      \qquad
       -   \inp{\xi_{(1)Y}e_i}{\xi_{(3)X} e_i}+   \inp{\xi_{(1)X}e_i}{\xi_{(3)Y} e_i}- (n-3) \langle \xi_{(1)\theta^\sharp} X,  Y\rangle 
             \\
      &
\qquad - \tfrac12 \langle \xi_{(2)X} e_i , \xi_{(3)Y} e_i \rangle  
         + \tfrac12  \langle \xi_{(2)Y} e_i , \xi_{(3)X} e_i \rangle + \tfrac{n}2 \langle \xi_{(2)\theta^\sharp} X,  Y\rangle.
         \end{split} 
  \end{equation}
 Another expression for $ \Ric^*_{\real{\lambda^{2,0}}}$ has been obtained in   \cite[Lemma 3.8]{Martin3}, it is given by  
$$
\Ric^*_{\real{\lambda^{2,0}}}(X,Y) =   \langle (\nabla^{\Lie{U}(n)}_{e_i} \xi)_{Je_i} JX, Y \rangle  - \langle \xi_{J \xi_{e_i} e_i} JX, Y \rangle.
$$
Now,  as before,  taking into account  properties of $\xi_{(i)}$, it is deduced  
\begin{equation} \label{Ricastdos}
    \begin{split}
          \Ric^*_{\lcf \lambda^{2,0}\rcf} (X,Y) 
     & 
     =
     -   \inp{(\Nt_{e_i}\xi_{(1)})_{e_i} X}{Y} 
     -    \inp{(\Nt_{e_i}\xi_{(2)})_{e_i} X}{Y} 
        +   \inp{(\Nt_{e_i}\xi_{(3)})_{e_i} X}{Y}
        \\
     &    
     \quad  +  \tfrac12 (d \theta)_{\lcf \lambda^{2,0} \rcf} (X,Y)  
    +     \tfrac{n-3}{2} \langle \xi_{(1)\theta^\sharp} X,  Y\rangle 
 + \tfrac{n}{2}     
\langle \xi_{(2)\theta^\sharp} X,  Y\rangle
\\
&
  \quad
  -   \tfrac{n-1}{2}    \langle \xi_{(3)\theta^\sharp} X,  Y\rangle.
         \end{split} 
  \end{equation}
If we take the difference between the identities \eqref{Ricastuno} and \eqref{Ricastdos}, we will obtain $\tfrac{n-2}{2} (d \theta)_{\lcf \lambda^{2,0} \rcf} - B$, where $B$ is the right side expression of the corresponding identity in Proposition \ref{divergenciaunouno}. This was already noted in \cite{Martin3} and explains why both expressions  agree.

Next, as consequences of the  identity \eqref{Ricastuno} (or \eqref{Ricastdos}) and those ones given in Proposition \ref{divergenciaunouno}, we display some relevant particular situations.  
\begin{proposition} \label{skewricac}
On almost Hermitian manifolds, relative to   $\Ricac_{\lcf \lambda^{2,0}\rcf}$, we have: 
\begin{enumerate}
\item[$(i)$] If $\xi \in \mathcal{W}_1 \oplus \mathcal{W}_2 \oplus \mathcal{W}_4$, then
\begin{equation*}
    \begin{split}
  \Ric^*_{\real{\lambda^{2,0}}}(X,Y)            &
                 =  -   \inp{(\Nt_{e_i}\xi_{(2)})_{e_i} X}{Y} 
	+ \tfrac{n+1}6 (d \theta)_{\lcf \lambda^{2,0} \rcf} (X,Y)
+ \tfrac{n}2 \langle \xi_{(2)\theta^\sharp} X,  Y\rangle. 
         \end{split} 
  \end{equation*}
 In particular:
 \begin{enumerate}
 \item[$(a)$]
  if $\xi \in   \mathcal{W}_1  \oplus \mathcal{W}_4$, then $\Ric^*_{\real{\lambda^{2,0}}}=  \tfrac{n+1}{6}  d\theta$. Moreover, if $n=3$, $\Ric^*_{\real{\lambda^{2,0}}}= d \theta = 0$.
   \item[$(b)$]
  if $\xi \in   \mathcal{W}_2  \oplus \mathcal{W}_4$ and $n>2$, then $$
  \Ric^*_{\real{\lambda^{2,0}}}(X,Y)= -   \inp{(\Nt_{e_i}\xi_{(2)})_{e_i} X}{Y}
+ \tfrac{n}2 \langle \xi_{(2)\theta^\sharp} X,  Y\rangle; 
$$
if $n=2$, then 
$$
\qquad \; \Ric^*_{\real{\lambda^{2,0}}}(X,Y) = -   \inp{(\Nt_{e_i}\xi_{(2)})_{e_i} X}{Y} 
	+ \tfrac{1}2 (d \theta)_{\lcf \lambda^{2,0} \rcf} (X,Y)
+  \langle \xi_{(2)\theta^\sharp} X,  Y\rangle.
$$
  \end{enumerate}

\item[$(ii)$]
If the structure is  Hermitian $( \xi \in \mathcal{W}_3 \oplus \mathcal{W}_4)$,  then:
\begin{enumerate}
\item[$(a)$] 
$\Ric^*_{\real{\lambda^{2,0}}}  =  \tfrac{n-1}{2}  d\theta_{\real{\lambda^{2,0}}}$; moreover, if $n>2$, we also have the expression 
$$
\Ric^*_{\real{\lambda^{2,0}}}   =   \tfrac{n-1}{n-2} \left( \langle (\nabla^{\Lie{U}(n)}_{e_i} \xi_{(3)})_{e_i} X, Y \rangle - \tfrac{n-1}{2} \langle \xi_{(3) \theta^\sharp } X, Y \rangle \right).
$$
\item[$(b)$]
If $n=2$, $(\Ric - \Ricac)_{[\lambda^{1,1}_0]} =0$.
\end{enumerate}
\end{enumerate}
\end{proposition}   
Part (ii) in Proposition  completes the result  given in \cite[Proposition 7.2]{Falcitelli-FS:aH}.

Another component  of the curvature in the orthogonal complement  $\mathcal{K}^\perp$ is determined by the  $\lcf \sigma^{2,0}\rcf$-component $\Ric_{\lcf \sigma^{2,0}\rcf} = (\Ric - \Ricac)_{\lcf \sigma^{2,0}\rcf}$ of $\Ric$. From \eqref{difric},   taking into account  properties of $\xi_{(i)}$ and the fact that $\Nt_X \theta$ is the  Lee form of $\Nt_X \xi$,  it is obtained 
\begin{equation*} 
    \begin{split}
       \Ric_{\lcf \sigma^{2,0} \rcf} (X,Y)
        & = -  \inp{(\Nt_{e_i}\xi_{(2)})_X Y}{e_i}    -  \inp{(\Nt_{e_i}\xi_{(2)})_Y X}{e_i}
        \\
        & 
        \quad - \tfrac{n-1}{4} ( (\nabla_X \theta)(Y)  +   (\nabla_Y \theta)(X) 
         -   (\nabla_{JX} \theta)(JY)    -   (\nabla_{JY} \theta)(JX))
         \\
      &
      \quad  
       +  \inp{\xi_{(1)X}e_i}{\xi_{(3)Y} e_i}
          +   \inp{\xi_{(1)Y}e_i}{\xi_{(3)X} e_i}
       - \tfrac12   \inp{\xi_{(2)X}e_i}{\xi_{(3)Y} e_i}
       \\
       &
       \quad
          - \tfrac12   \inp{\xi_{(2)Y}e_i}{\xi_{(3)X} e_i} + \tfrac{n-2}{2} \theta (\xi_{(2)X} Y + \xi_{(2)Y} X). 
               \end{split}
  \end{equation*}
  As a consequence of this expression we have the following result. 
\begin{proposition}\label{ricsigma}
If an almost Hermitian manifold is such that $\xi \in \mathcal{W}_1 \oplus \mathcal{W}_4$ or $\xi \in \mathcal{W}_3 \oplus \mathcal{W}_4$, then   
\begin{equation*} 
    \begin{split}
       \Ric_{\lcf \sigma^{2,0} \rcf} (X,Y)
        & = - \tfrac{n-1}{4} ( (\nabla_X \theta)(Y)  +   (\nabla_Y \theta)(X) 
         -   (\nabla_{JX} \theta)(JY)    -   (\nabla_{JY} \theta)(JX)).
               \end{split}
  \end{equation*}
Moreover,  in such cases, if  the vector field $\theta^\sharp$ is Killing, then   $\Ric_{\lcf \sigma^{2,0} \rcf}=0$.
\end{proposition}

\subsection{Ricci forms}
 We would like to  focus our  attention to components of the curvature included in the space $\mathcal{K}$ of type  Kähler curvature  tensors. Such an space  is decomposed into $\mathcal{K} = \mathcal{C}_3 \oplus \mathcal{K}_1 \oplus 
\mathcal{K}_2$, where $\mathcal{C}_3 \cong [\sigma^{2,2}_0]$,  $\mathcal{K}_1 \cong \mathbb{R}$ and  $\mathcal{K}_2 \cong [\lambda^{1,1}_0]$ (see \cite{Falcitelli-FS:aH}). The components in $\mathcal{K}_1$ and in $\mathcal{K}_2$ can be determined in terms of tensor  $\Ric$ and $\Ricac$. More precisely, they are obtained by the tensor  $(\Ric + 3 \Ricac)_{[\lambda^{1,1}]}$. To derive expressions for such a tensor,  
we consider convenient to  previously  recall some notions relative to Ricci forms (see \cite{Gau}). 
\begin{definition}
{\rm For  almost Hermitian manifolds, given a metric connection $D$ with curvature tensor $R_D$, the  \textit{first
 $($second$)$  Ricci form}  of $D$ is the two-form $\rho_D$ $(r_D)$ given by 
$$
\rho_D(X,Y) = - \tfrac12 \langle R_D(e_i,Je_i) X, Y \rangle \qquad r_D(X,Y) = - \tfrac12 \langle R_D(X,Y) e_i, J e_i \rangle.
$$ }
\end{definition}

Next we compute the Ricci forms relative to  the Levi Civita  and the minimal connections.
\begin{proposition} \label{riccigeneral}
For  almost Hermitian manifolds, we have:

\noindent $(i)$ If $D$ is a $\Lie{U}(n)$-connection, then $\rho_D \in [\lambda^{1,1}]$ and  $r_D$  is closed.
\begin{align*}
\hspace{-6mm} (ii) \;\,   \Ricac(X,JY)  = &\; \rho_\nabla (X,Y) =  r_\nabla(X,Y) =  r_{\nabla^{\Lie{U}(n)}}(X,Y) +  \langle \xi_{X} e_i , \xi_{Y} Je_i \rangle \qquad \; \qquad \;\;\quad
\\
= & \; \textstyle  r_{\nabla^{\Lie{U}(n)}}(X,Y) 
+ \sum_{a=1}^3 \langle \xi_{(a)X} e_i , \xi_{(a)Y} Je_i \rangle 
 - \textstyle \sum_{a=1}^3 J\theta ( \xi_{(a)X} Y - \xi_{(a)Y} X)
\\
 & 
\textstyle
+ \sum_{1 \leq a < b \leq 3} (\langle \xi_{(a)X} e_i , \xi_{(b)Y} Je_i \rangle - \langle \xi_{(a)Y} e_i , \xi_{(b)X} Je_i \rangle)
\\
&
 -  \tfrac14 \|\theta\|^2  \omega(X,Y)
- \tfrac14 \theta \wedge J \theta (X,Y)
\end{align*}
\begin{align*}
\hspace{-1mm}  (iii) \; \,\rho_{\nabla [\lambda^{1,1}]} (X,Y)
  = & \;
 \rho_{\nabla^{\Lie{U}(n)}}(X,Y) 
 +  \langle \xi_{e_i} X , \xi_{Je_i} Y \rangle 
 \\
  = & \;
  \rho_{\nabla^{\Lie{U}(n)}}(X,Y) 
 + \textstyle \sum_{a=1}^3   \langle \xi_{(a)e_i} X , \xi_{(a)Je_i} Y \rangle
 - \tfrac18 \|\theta \|^2 \omega(X,Y) \qquad \; \qquad \;\;\quad
\\
 & \;
  \textstyle + \sum_{1 \leq a < b \leq 3 } ( \langle \xi_{(a)e_i} X , \xi_{(b)Je_i} Y \rangle -  \langle \xi_{(a)e_i} Y , \xi_{(b)Je_i} X \rangle)
   \\
 &
 - \tfrac12 J\theta(\xi_{(3)X} Y - \xi_{(3)Y} X)   + \tfrac{n-2}{8} \theta \wedge J\theta(X,Y).
 \end{align*}
\end{proposition}
\begin{proof} If we consider an adapted local frame $\wp = \{e_\alpha, Je_\alpha = e_{\alpha'}\}$ to the $\Lie{U}(n)$-structure, $\alpha' = \alpha + n$ and $\alpha=1, \dots, n$, then 
$
r_D(X,Y) = \textstyle \sum_{\alpha=1}^n \Omega^{D\,\alpha}_{\; \;\; \;\alpha'}(\wp_*X,\wp_*Y),
$ 
 where $\Omega^D = \left(\Omega^{D\,i}_{\; \;\; \; j}\right)$ is the curvature two-form of $D$. If $D$ is a $\Lie{U}(n)$-connection and $\varpi^D = \left(\varpi^{D\,i}_{\; \;\; \;j}\right)$ is the connection  one-form of $D$, then one has $\varpi^{D\,\alpha'}_{\; \;\; \;\beta'} =  \varpi^{D\,\alpha}_{\; \;\; \;\beta} = -\varpi^{D\,\beta}_{\; \;\; \;\alpha} $ and $\varpi^{D\,\alpha'}_{\; \;\; \;\beta} = - \varpi^{D\,\alpha}_{\; \;\; \;\beta'} =\varpi^{D\,\beta'}_{\; \;\; \;\alpha} $. Hence
\begin{eqnarray*}
\textstyle\sum_{\alpha=1}^n \sum_{i=1}^{2n} \varpi^{D\,i}_{\; \;\; \;\alpha'} \wedge \varpi^{D\,\alpha}_{\; \;\; \;i} 
& = & \textstyle \sum_{\alpha,\beta=1}^n (\varpi^{D\,\beta}_{\; \;\; \;\alpha'} \wedge \varpi^{D\,\alpha}_{\; \;\; \;\beta} + \varpi^{D\,\beta'}_{\; \;\; \;\alpha'} \wedge \varpi^{D\,\alpha}_{\; \;\; \;\beta'}) \\
& = & \textstyle \sum_{\alpha,\beta=1}^n (-  \varpi^{D\,\alpha}_{\; \;\; \;\beta'} \wedge \varpi^{D\,\beta}_{\; \;\; \;\alpha} - \varpi^{D\,\alpha'}_{\; \;\; \;\beta'} \wedge \varpi^{D\,\beta}_{\; \;\; \;\alpha'}) = 0.
\end{eqnarray*}
Now using the structure equation, one has
$$
\textstyle \sum_{\alpha=1}^n \Omega^{D\,\alpha}_{\; \;\; \;\alpha'} = \sum_{\alpha=1}^n d  \varpi^{D\,\alpha}_{\; \;\; \;\alpha'} - \sum_{\alpha=1}^n \sum_{i=1}^{2n} \varpi^{D\,i}_{\; \;\; \;\alpha'} \wedge \varpi^{D\,\alpha}_{\; \;\; \;i} =  \sum_{\alpha=1}^n d  \varpi^{D\,\alpha}_{\; \;\; \;\alpha'}.
$$
From this, it follows that $r_D = \sum_{\alpha=1}^n  \wp^*\Omega^{D\,\alpha}_{\; \;\; \;\alpha'} = d  \left( \textstyle \sum_{\alpha=1}^n \wp^*  \varpi^{D\,\alpha}_{\; \;\; \;\alpha'}\right)$ is closed.

On the other hand, if $D$ is a $\Lie{U}(n)$-connection, then the matrix 
$$
\textstyle \left(\rho_D(e_i , e_j) \right)  =  \left( \sum_{\alpha = 1}^n \Omega^{D\,j}_{\;\;\;i}(\wp_\ast e_\alpha, \wp_\ast J e_\alpha)\right) 
$$
belongs to the Lie algebra $\lie{u}(n)\cong [\lambda^{1,1}]$ of $\Lie{U}(n)$.  

Since  $D$ is a metric connection, then  $D = \nabla + \xi^D$, where $\langle \xi^D_X Y , Z \rangle =  -\langle \xi^D_X Z , Y \rangle $. Moreover, the corresponding curvature tensors $R$ and $R^D$ are related by 
\begin{equation} \label{RRnabla}
R(X,Y) Z = R^D(X,Y) Z +  (D_X \xi^D)_Y Z -  (D_Y \xi^D)_X Z +  \xi^D_{\xi^D_XY} Z -  \xi^D_{\xi^D_YX} Z 
- [\xi^D_X , \xi^D_Y] Z.
\end{equation}
Therefore,
\begin{eqnarray} \label{rrD}
-2 r_\nabla(X,Y) & = &  -2 r_D(X,Y) + \langle (D_X \xi^D)_Y e_i, Je_i \rangle  -  \langle (D_Y \xi^D)_X e_i, Je_i \rangle \\
&&+ \langle   \xi^D_{\xi^D_XY - \xi^D_YX} e_i , J e_i \rangle - 2\langle \xi^D_X e_i , \xi^D_Y Je_i \rangle. \nonumber
\end{eqnarray}
In particular, when $D = \nabla^{\Lie{U}(n)}$, we will have
\begin{eqnarray*}
 r_\nabla(X,Y) & = &  r_{\nabla^{\Lie{U}(n)}}(X,Y) + \langle \xi_X e_i , \xi_Y Je_i \rangle. 
\end{eqnarray*}
From this, by using the properties of $\xi_{(a)}$, it follows the other  expression for $ r_\nabla$  in (ii).

For (iii), from \eqref{RRnabla} we have
\begin{equation*} 
 \rho_\nabla (X,Y)   =  \rho_{\nabla^{\Lie{U}(n)}}(X,Y)  -   \langle (\nabla^{\Lie{U}(n)}_{e_i} \xi)_{Je_i} X, Y \rangle   + \tfrac{n-1}{2} \langle  \xi_{J \theta } X, Y \rangle  + \langle \xi_{e_i} X , \xi_{Je_i} Y \rangle. 
\end{equation*}
Because $\nabla^{\Lie{U}(n)}$ is a $\Lie{U}(n)$-connection, 
$
 \rho_{\nabla \, [\lambda^{1,1}]} (X,Y)   =  \rho_{\nabla^{\Lie{U}(n)}}(X,Y)   + \langle \xi_{e_i} X , \xi_{Je_i} Y \rangle$. 
 Now, by using the properties of $\xi_{(a)}$, it follows the other   expression  for $ \rho_{\nabla [\lambda^{1,1}]}$ in (iii).
\end{proof}

\begin{remark}
{\rm  Note that in  case of a presence of an $\Lie{SU}(n)$-structure by \cite[Lemma 3.3]{Martin3}, one  has $r_{\nabla^{\Lie{U}(n)}} = - n d \widehat{\eta} $, where 
$\nabla^{\Lie{SU}(n)} =  \nabla^{\Lie{U}(n)} + \eta$ and $\eta_X Y = \widehat{\eta}(X) JY$, i.e. $\widehat{\eta} = J \eta$. Hence $r_{\nabla^{\Lie{U}(n)}}$  would be exact. Anyway, since  local $\Lie{SU}(n)$-structures always exist on an almost Hermitian manifold, it follows that $r_{\nabla^{\Lie{U}(n)}}$ is locally the exterior derivative  of a local one-form. This is an alternative argument confirming that $r_{\nabla^{\Lie{U}(n)}}$ is closed. 

 The first Chern class $c_1$ can be represented by $- \tfrac1{2\pi} r_{D}$ (in the de Rham cohomology group), where $D$ is any $\Lie{U}(n)$-connection. In particular, it is   by $- \tfrac1{2\pi} r_{\nabla^{\Lie{U}(n)}}$.  Note that in case of Kähler manifold, 
 $\nabla^{\Lie{U}(n)}  = \nabla$ and $r_{\nabla}(X,Y) = \Ric(X,JY)$. This has motivated the name  \textquoteleft Ricci form\textquoteright .
 On other hand, the first Chern class $c_1$ is associated to the tangent bundle $TM$ considered as a complex vector bundle (by means of  $J$). Thus $c_1(M) = c_1(T_\mathbb{C}M=TM) = - c_1(\Lambda^n (T_\mathbb{C}M)^*)$. The vanishing of $c_1$ is a necessary  condition for the existence of a complex volume form globally defined on $M$, i.e. existence of a $\Lie{SU}(n)$-structure. 
 }
\end{remark}

 In complex geometry there is a $\Lie{GL}(n, \mathbb{C})$-connection which plays a relevant a role, the \textit{Chern connection}. It is the unique   $\Lie{GL}(n, \mathbb{C})$-connection $\nabla^h$  such that its  torsion $T^h$ satisfies $T^h(JX,Y) = JT^h(X,Y)$. In Hermitian geometry $(\xi \in \mathcal{W}_3 \oplus \mathcal{W}_4)$, 
the Chern connection is given by $\nabla^h = \nabla + \xi^h$, where $\langle \xi^h_X Y , Z \rangle = \langle \xi_X Y + \xi_Y X, Z \rangle -  \langle \xi_Z X, Y \rangle$.  In almost Hermitian geometry, it is straightforward to check that $\nabla^h$ defined as before is a $\Lie{U}(n)$-connection if and only if $\xi \in \mathcal{W}_3 \oplus \mathcal{W}_4$. 

\begin{proposition} \label{RicciChern}
For  Hermitian manifolds, we have:\vspace{1mm}

\noindent $(i)$ $r_{\nabla^h} = r_{\nabla^{\Lie{U}(n)}} + \tfrac{n-1}{2} d J\theta$ . Moreover, $r_{\nabla^h} = r_{\nabla^{\Lie{U}(n)}[\lambda^{1,1}]} +  \tfrac{n-1}{2} (d J\theta)_{[\lambda^{1,1}]}$.  
\begin{align*} 
\hspace{-3mm} (ii) \; \, \rho_{\nabla^h}(X,Y)  = &   
\;  \rho_{\nabla[\lambda^{1,1}]}(X,Y) 
 -  \langle (\nabla^{\Lie{U}(n)}_{e_i} \xi_{(3)})_X Y , {J e_i}  \rangle 
 +  \langle (\nabla^{\Lie{U}(n)}_{e_i} \xi_{(3)})_Y X, {J e_i}  \rangle   
  \\
 &
  - \tfrac12 (d J \theta)_{[\lambda^{1,1}]} (X,Y) 
 + \tfrac12 d^*\theta \, \omega(X,Y) 
 + \tfrac{2n-1}{4} \|\theta \|^2 \omega(X,Y)
 + \tfrac14 \theta \wedge J \theta(X,Y)
 \\
 &
 + \tfrac{n}{2}  J \theta (\xi_{(3)X} Y - \xi_{(3)Y} X) 
- 2 \langle \xi_{(3)e_i} X  , \xi_{(3)Je_i} Y \rangle +   \langle \xi_{(3)X} e_i   , \xi_{(3)Y} {Je_i}  \rangle.
\end{align*}
\end{proposition}
\begin{remark}{\rm
Because the difference $ r_{\nabla^h} - r_{\nabla^{\Lie{U(n)}}}$ is an  exact two-form,  the first Chern class  is determined by $-\tfrac{1}{2\pi} r_{\nabla^h}$ or  by $ -\tfrac{1}{2\pi} r_{\nabla^{\Lie{U(n)}}}$   as it is expected.}
\end{remark}
\begin{proof}
We will use  \eqref{rrD}. Thus, we obtain the following identities by direct computation
\begin{gather*}
\langle (\nabla^h_X \xi^h)_Y e_i ,  Je_i\rangle - \langle (\nabla^h_Y \xi^h)_X e_i ,  Je_i\rangle  = 
  (n-1) dJ\theta(X,Y) + 2 (n-1) \langle \xi_{J\theta} X, Y \rangle, \\
\langle \xi^h_{\xi^h_XY} e_i , J e_i \rangle - \langle \xi^h_{\xi^h_YX} e_i , J e_i \rangle  =  - 2 (n-1) \langle \xi_{J \theta}X,Y\rangle, \qquad 
 \langle \xi^h_X e_i, \xi^h_Y J e_i \rangle  =  \langle \xi_X e_i, \xi_Y J e_i \rangle. 
\end{gather*}
Hence we have 
$
r_\nabla(X,Y) = r_{\nabla^h}(X,Y) - \tfrac{n-1}{2} d J \theta + \langle \xi_X e_i, \xi_Y J e_i \rangle
$
 Now using Proposition \ref{riccigeneral} (ii), we get $r_{\nabla^h} = r_{\nabla^{\Lie{U}(n)}} + \tfrac{n-1}{2} d J \theta$. Finally, since $\xi \in \mathcal{W}_3 \oplus \mathcal{W}_4$, using Proposition \ref{skewricac} (ii) and the properties of $\xi$, we obtain 
$$ 
r_{\nabla^{\Lie{U}(n)}\lcf \lambda^{2,0} \rcf}(X,Y) = r_{\nabla\lcf \lambda^{2,0} \rcf}(X,Y) =  \tfrac{n-1}{2} (d\theta)_{\lcf \lambda^{2,0}\rcf}(X,JY) = -  
 \tfrac{n-1}{2} (dJ\theta)_{\lcf \lambda^{2,0}\rcf}(X,Y).   
$$
Then $r_{\nabla^h}= r_{\nabla^{\Lie{U}(n)}[\lambda^{1,1}]} + \tfrac{n-1}{2} (dJ\theta)_{[ \lambda^{1,1}]}(X,Y)$. 

For (ii), it is used the identity \eqref{RRnabla} for $\nabla^h$, the facts that $\rho_{\nabla^h}$ is in $[\lambda^{1,1}]$ and $\nabla^h = \nabla^{\Lie{U}(n)} - \xi + \xi^h$, the definition of $\xi^h$ and the properties of the components $\xi_{(i)}$ of $\xi$.  
\end{proof}

\subsection{Some components in the space of Kähler curvatures}
As it was mentioned before, we will derive expressions for the tensor $(\Ric + 3 \Ricac)_{[\lambda^{1,1}]}$. This tensor determines the components of the curvature in $\mathcal{K}_1\cong \mathbb{R}$ and $\mathcal{K}_2 \cong [\lambda^{1,1}_0]$. 
\begin{proposition} \label{unomastresric}
For almost hermitian manifolds, we have
\begin{equation*} 
    \begin{split}
    \tfrac12  ( \Ric + 3 \Ricac )_{[\lambda^{1,1}]} (X,Y)
      &  =  - 2 \textstyle  r_{\nabla^{\Lie{U}(n)}[\lambda^{1,1}]}(X,JY) 
          -  \inp{(\Nt_{e_i}\xi_{(3)})_X Y}{e_i} 
          \\
          &
          \quad \,     
          - \tfrac{n-2}{4} ( ( \nabla_X \theta ) ( Y)  + ( \nabla_{JX} \theta ) ( JY))
          + \tfrac14(d^* \theta + \tfrac{2n-7}2 \|\theta \|^2 )  \langle X, Y \rangle 
  \\
  &\quad \,     
   -  2 \inp{\xi_{(2) X} e_i}{\xi_{(2)Y}e_i} 
    -  \langle \xi_{(2) e_i} X , \xi_{(2) e_i } Y \rangle 
     + 2 \langle \xi_{(3)X} e_i , \xi_{(3)Y} e_i \rangle
\\
&
\quad \,  
  - \tfrac{n-6}{8} (\theta(X) \theta(Y) + \theta(JX) \theta(JY))
  +  \inp{\xi_{(1) X} e_i}{\xi_{(2)Y}e_i}
  \\
  &
 \quad \,   + \tfrac52  \inp{\xi_{(1)Y}  e_i}{\xi_{(2)X}e_i} 
+\tfrac{n-6}{2} \theta (\xi_{(3)X} Y) - 2 \theta (\xi_{(3)Y} X).
                  \end{split}
  \end{equation*}
\end{proposition}
\begin{proof}
This identity directly follows  by computing $4\Ricac_{[\lambda{1,1}]}
 + (\Ric - \Ricac)_{[\lambda^{1,1}]}$, using Proposition \ref{riccigeneral} (ii) and the identity \eqref{ricminusricacuno}.
\end{proof}

The $\mathcal{K}_1$-component of the curvature is determined by the metric contraction of $\Ric + 3 \Ricac$.
\begin{corollary} \label{unomastrescor}
 For almost Hermitian manifolds we have
$$
s + 3 s^* =  8 \langle   r_{\nabla^{\Lie{U}(n)}} , \omega \rangle  + 2 (n-1) d^* \theta + (n-3)(n-1)  \|\theta \|^2 - 6 \| \xi_{(2)} \|^2 + 4 \| \xi_{(3)} \|^2 
$$
\end{corollary} 
Thus, from the identity \eqref{difsca} and the previous one, the following expressions for the scalar curvatures are obtained
\begin{eqnarray*}  
s & =  & 2 \langle   r_{\nabla^{\Lie{U}(n)}} , \omega \rangle + 2 (n-1) d^* \theta + \tfrac12 (2n-3)(n-1)  \|\theta \|^2 + 3 \| \xi_{(1)} \|^2 - 3 \| \xi_{(2)} \|^2 + \| \xi_{(3)} \|^2,\\  
s^* \hspace{-1mm} & = & 2  \langle   r_{\nabla^{\Lie{U}(n)}} , \omega \rangle -   \tfrac12 (n-1)  \|\theta \|^2 - \| \xi_{(1)} \|^2 -  \| \xi_{(2)} \|^2 + \| \xi_{(3)} \|^2. 
\end{eqnarray*}

\section{Examples}
In this section we will display some examples showing that the components $d \theta_{[\lambda^{1,1}_{0}]}$ and $d \theta_{\real{\lambda^{2,0}}}$ of $d\theta$ can be non-zero. Also these examples will illustrate the formulae proved in the previous Section. For sake of simplicity, we will denote the wedge product by just  juxtaposition of superindices, i.e. $e^{ij} = e^i \wedge e^j$.  We also note that our convention for Nijenhuis tensor is $N(X,Y) = [X,Y] + J[JX,Y] + J[X,JY] -[JX,JY]$. We recall that the vanishing of $N$ characterizes the type $\mathcal{W}_3 \oplus \mathcal{W}_4$, called Hermitian structure.

 \begin{example}{\rm Let $G$ be the four-dimensional simply-connected real solvable Lie group 
 determined by the  Lie algebra, displayed in \cite{Graaf} and  denoted by $M^3_0$ there, generated by  the basis 
$\{e^1, \ldots , e^4\}$  of left-invariant one-forms  such that
$$
de^1 =   e^{14}, \quad de^2  = 0, \quad d e^3 =  e^{24} + e^{34}, \qquad  d e^4 =0. 
$$

On $G$ we consider the almost Hermitian structure such that $\langle \cdot , \cdot  \rangle = \sum_{i=1}^4 e^i \otimes e^i$  is the metric  and $\omega= e^{31} + e^{42}$ is the  Kähler form.

It is straightforward to check that the Nihenjuis tensor of the almost complex structure $J$ is zero. On the other hand, one has the exterior derivative 
$d \omega   =   (-e^1 -2 e^4)  \wedge \omega$. Therefore, the Lee form is $\theta=  -e^1 -2 e^4$ and we are in the presence of a Hermitian structure (type $\mathcal{W}_4$ in the four-dimensional case). The non-zero components of $d\theta$ are given by 
$$
(d \theta)_{[\lambda^{1,1}_0]} = - \tfrac12 (e^{14} + e^{23}), \qquad (d \theta)_{\real{\lambda^{2,0}}} = -\tfrac12 (e^{14} - e^{23}).    
$$
 Because $n=2$, the identities given in Proposition \ref{divergenciaunouno} are trivial. For dimension $4$, one has  $\mathcal{W}_1 \oplus \mathcal{W}_4 = \mathcal{W}_4$ and, in this example, $d \theta_{[\lambda^{1,1}_0]} \neq0$ and  $(d \theta)_{\real{\lambda^{2,0}}} \neq 0$. 
Finally, we have 
$$
s - s^* = 2 d^*\theta + \| \theta \|^2 = 9, \qquad  (\Ric - \Ric^*)_{\mathbb{R}} = \tfrac94 \langle \cdot , \cdot \rangle, \qquad    (\Ric - \Ric^*)_{[\lambda^{1,1}_0]} =0,
$$
$$
\Ric^*_{\lcf \lambda^{2,0}\rcf} = \tfrac12 (d \theta)_{\real{\lambda^{2,0}}} = -\tfrac14 (e^{14} - e^{23}).
$$
If  $\ast$ is the Hodge star operator with respect to $\mathrm{Vol} = - \frac12 \omega^2= e^{1234}$, $d^* \theta = - \ast d \ast \theta=  2$.
Now, we fix the complex volume form $\psi_+ + i \psi_-$, where $\psi_+ = e^{12} - e^{34}$ and $\psi_- = e^{32} + e^{14}$. For this $\Lie{SU}(2)$-structure, we have
$$
4  \eta   - \theta = \ast ( \ast d\psi_+ \wedge \psi_+ + \ast d\psi_- \wedge \psi_-) = 2 e^4.  
$$
Then $\widehat{\eta} = J \eta =  - \tfrac14 e^3$ and it is globally defined. Thus, $r_{\nabla^{\Lie{U}(2)}} = - 2 d \widehat{\eta} =  \tfrac12 (e^{24} + e^{34})$ is exact. Hence, as it is expected by the existence of the $\Lie{SU}(2)$-structure, the first Chern class  vanishes $c_1=0$.  The Ricci form 
 of the Chern connection is given by 
$
r_{\nabla^{h}} = r_{\nabla^{\Lie{U}(2)}} + \tfrac12 d J\theta = 0$.
The Lie brackets are given by  $[e_1, e_4] = -e_1$, $[e_2, e_4] = -e_3$, $[e_3, e_4] = -e_3$ and $[e_i, e_j] = 0$, for the remaining pairs $(i,j)$. By  Koszul's formula, the Levi Civita covariant derivatives are derived
\begin{gather*}
\nabla_{e_1} e_1 = \nabla_{e_3} e_3 = e_4, \quad \nabla_{e_1} e_4 = -e_1, \quad  \nabla_{e_2} e_3 = \nabla_{e_3} e_2 = \tfrac12 e_4,\\
\nabla_{e_2} e_4 = - \nabla_{e_4} e_2 = - \tfrac12 e_3, \quad \nabla_{e_3} e_4  = -\tfrac12 e_2 - e_3, \quad \nabla_{e_4} e_3 = - \tfrac12 e_2 
\end{gather*}
and $\nabla_{e_i} e_j =0$, for the remaining pairs $(i,j)$. These are needed in  Proposition \ref{ricsigma} to obtain the following component of the Ricci tensor
$$
\Ric_{\lcf 	\sigma^{2,0} \rcf} = \tfrac14 ( e^1 \otimes e^4 + e^4 \otimes e^1 + e^2 \otimes e^3 + e^3 \otimes e^2).
$$

Since $r_{\nabla^{\Lie{U}(2)}[\lambda^{1,1}]} =  \tfrac12 e^{24} + \tfrac14 (e^{34} + e^{12})$ and we have all the information required in Proposition \ref{unomastresric} to compute  the following Kähler component of the curvature   
\begin{eqnarray*}
(\Ric + 3 \Ric^*)_{[\lambda^{1,1}]} & = & 
  - \tfrac{7}{4} e^1 \otimes e^1 - \tfrac{3}{4} e^2 \otimes e^2 
 - \tfrac{7}{4} e^3 \otimes e^3 - \tfrac{3}{4} e^4 \otimes e^4
 \\
 && 
 +3 e^1 \otimes e^4 + 3 e^4 \otimes e^1 - 3  e^2 \otimes e^3 -  3 e^3 \otimes e^2.
\end{eqnarray*}
As a consequence (or by Corollary \ref{unomastrescor}), for the scalar curvatures it is obtained
$$
s+ 3 s^* = -5, \qquad  s = \tfrac{11}{2}, \qquad s^* = - \tfrac{7}{2}.
$$

Since one has the identities $\rho_\nabla = r_\nabla = - J_{(2)} \Ricac$, we obtain
$$
 \rho_\nabla =  r_\nabla =  \Ricac(\cdot , J \cdot)= -  e^{31} - \tfrac{3}{4} e^{42} + \tfrac12  e^{12} + \tfrac32 e^{34} 
$$
This form is not closed, 
$
d \rho_\nabla = - 2 e^{134} -  \tfrac32  e^{124} 
$.  Now  we use Proposition \ref{riccigeneral} (iii), to compute the first Ricci form of $\nabla^{\Lie{U}(2)}$ which is given by 
$$
 \rho_{\nabla^{\Lie{U}(2)}}  = - \tfrac38  e^{31} - \tfrac18  e^{42} + \tfrac34 e^{12} + \tfrac34 e^{34}.
$$
It is  in $[\lambda^{1,1}]$ but its exterior derivative is non-zero, $d \rho_{\nabla^{\Lie{U}(2)}}  = - \tfrac98  e^{124}- \tfrac34  e^{134} \neq 0 $. Likewise, by using  Proposition \ref{RicciChern} (ii),  the first Ricci form 
of  the Chern connection is given by 
$$
\rho_{\nabla^h} = \tfrac72  e^{31}  -\tfrac52  e^{42}  + \tfrac12  e^{12} + \tfrac12  e^{34}.
$$
It is also in $[\lambda^{1,1}]$ and has non-zero exterior derivative,  $d \rho_{\nabla^h}  = 3 e^{124} + \tfrac72  e^{134} \neq 0 $.
}
 \end{example}

\begin{example}{\rm 
In  \cite{Hasegawa}, Hasegawa determined all the complex surfaces  which are diffeomorfic to compact solvmanifolds. Some of them are known as Inoue surfaces \cite{Inoue} of type $S^+$ or $S^-$. These can be written (up to finite covering) as $\Gamma \backslash G$, where $\Gamma$ is a lattice of a simply connected solvable Lie group $G$ (see   \cite{Hasegawa} for details). The Lie algebra $\mathfrak{g}$ of $G$ is expressed as having a basis $\{ X_1, X_2, X_3, X_4\}$ with the bracket multiplication
$
[X_2, X_3] = - X_1$,  $[X_4, X_2] =  X_2$, $[X_4, X_3] = - X_3
$
and all other brackets vanish. The almost complex structure $J$ is defined by 
$$
JX_1 = X_2, \quad -JX_2 = X_1, \quad JX_3 = X_4 - qX_2, \quad JX_4 = -X_3 - qX_1, \qquad q \in \mathbb R,
$$
for which the Nijenhuis tensor  vanishes. We will consider the metric $\langle \cdot , \cdot \rangle$ that makes orthonormal the basis
$$
e_1 = X_1, \quad e_2 = X_2, \quad e_3 = X_3, \quad e_4 = X_4- q X_2.
$$
Such a  metric is compatible with the complex structure $J$, then we are in the presence of a Hermitian structure (type $\mathcal{W}_4$ in the four-dimensional case). For the basis $\{ e_1, e_2, e_3, e_4\}$, one has
$$
[e_2, e_3] = - e_1, \quad [e_2, e_4] =  - e_2, \qquad [e_3, e_4] =  - q e_1+ e_3 ,
$$ 
and all other brackets vanish. The corresponding basis 
$\{e^1, e^2, e^3 , e^4\}$  of left-invariant one-forms  are such that
$$
de^1 =   e^{23} +q e^{34}, \quad de^2  = e^{24}, \quad d e^3 =  - e^{34}, \qquad  d e^4 =0. 
$$ 
The Hermitian metric is expressed as  $\langle \cdot , \cdot  \rangle = \sum_{i=1}^4 e^i \otimes e^i$  and the Kähler two-form  as  $\omega= e^{21} + e^{43}$. Its exterior derivative is given by 
$d \omega   =   ( q e^2 -e^4)  \wedge \omega$. Therefore, the Lee form is $\theta= q e^2 - e^4 $. The non-zero components of $d\theta$ are given by 
$$
(d \theta)_{[\lambda^{1,1}_0]} =  \tfrac{q}2 (e^{13} + e^{24}), \qquad (d \theta)_{\real{\lambda^{2,0}}} = \tfrac{q}2 (-e^{13} + e^{24}).    
$$
 Because $n=2$, the identities given in Proposition \ref{divergenciaunouno} are trivial. For dimension $4$, one has  $\mathcal{W}_1 \oplus \mathcal{W}_4 = \mathcal{W}_4$ and, in this example, $d \theta_{[\lambda^{1,1}_0]} \neq0$ and  $(d \theta)_{\real{\lambda^{2,0}}} \neq 0$. 
Finally, we have 
$$
s - s^* = 2 d^*\theta + \| \theta \|^2 =  1 + q^2, \qquad  (\Ric - \Ric^*)_{\mathbb{R}} = \tfrac{1+q^2}4 \langle \cdot , \cdot \rangle, \qquad    (\Ric - \Ric^*)_{[\lambda^{1,1}_0]} =0,
$$
$$
\Ric^*_{\lcf \lambda^{2,0}\rcf} = \tfrac12 (d \theta)_{\real{\lambda^{2,0}}} = \tfrac{q}4 (-e^{13} + e^{24}).
$$
If  $\ast$ is the Hodge star operator with respect to $\mathrm{Vol} = - \frac12 \omega^2= - e^{1234}$, $d^* \theta = - \ast d \ast \theta=  0$.
Now, we fix the complex volume form $\psi_+ + i \psi_-$, where $\psi_+ = e^{13} - e^{24}$ and $\psi_- = e^{14} + e^{23}$. For this $\Lie{SU}(2)$-structure, we have
$$
4  \eta   - \theta = \ast ( \ast d \psi_+ \wedge \psi_+ + \ast d\psi_- \wedge \psi_-) = - 2 e^4. 
$$
Then $\widehat{\eta} = J \eta =  \tfrac14 ( -q e^1 +  3e^3)$ and it is globally defined. Thus, $r_{\nabla^{\Lie{U}(2)}} = - 2 d \widehat{\eta} =  \tfrac{q}2 q e^{23}  +\tfrac{3+q^2}{2} e^{34}$ is exact. Hence, as it is expected by the existence of the $\Lie{SU}(2)$-structure, the first Chern class $c_1$  vanishes.  The Ricci form 
 of the Chern connection is given by 
$
r_{\nabla^{h}} = r_{\nabla^{\Lie{U}(2)}} + \tfrac12 d J\theta =   e^{34}$.
 By  Koszul's formula, the Levi Civita covariant derivatives are derived
\begin{gather*}
\nabla_{e_2} e_2 = -\nabla_{e_3} e_3 = e_4, \quad \nabla_{e_1} e_2 = \nabla_{e_2} e_1 = \tfrac{1}2 e_3, \quad  \nabla_{e_1} e_3 = \nabla_{e_3} e_1 = -\tfrac12  e_2+\tfrac12 q e_4,\\
 \nabla_{e_1} e_4 = \nabla_{e_4} e_1 = -\tfrac12 q e_3, \quad \nabla_{e_2} e_3 = - \nabla_{e_3} e_2 = - \tfrac12 e_1, \quad \nabla_{e_3} e_4  = -\tfrac{q}2 e_1 + e_3, \quad \nabla_{e_4} e_3 =  \tfrac{q}2 e_1 
\end{gather*}
and $\nabla_{e_i} e_j =0$, for the remaining pairs $(i,j)$. These are needed in  Proposition \ref{ricsigma} to obtain the following component of the Ricci tensor
$$
\Ric_{\lcf 	\sigma^{2,0} \rcf} = \tfrac12 ( e^1 \otimes e^1 - e^2 \otimes e^2 + e^3 \otimes e^3 - e^4 \otimes e^4).
$$

Since $r_{\nabla^{\Lie{U}(2)}[\lambda^{1,1}]} =  \tfrac{3+q^2}{2} e^{34} + \tfrac{q} 4 ( e^{23} -  e^{14})$ and we have all the information required in Proposition \ref{unomastresric} to compute  the following Kähler component of the curvature   
\begin{eqnarray*}
(\Ric + 3 \Ric^*)_{[\lambda^{1,1}]} & = & \tfrac{-3 + q^2}{4}  (e^1 \otimes e^1 +  e^2 \otimes e^2) - \tfrac{23 +11q^2}{4}  
( e^3\otimes e^3 +  e^4 \otimes e^4). 
\end{eqnarray*}
As a consequence (or by Corollary \ref{unomastrescor}), for the scalar curvatures it is obtained
$$
s+ 3 s^* = -13- 5 q^2, \qquad  s = - \tfrac{5+q^2}{2}, \qquad s^* = - \tfrac{7+ 3q^2}{2}.
$$

Since one has the identities $\rho_\nabla = r_\nabla = - J_{(2)} \Ricac$, we obtain
$$
 \rho_\nabla =  r_\nabla =  \Ricac(\cdot , J \cdot)= \tfrac14   e^{12} + \tfrac{8+ 3q^2}{4} e^{34} 
+ \tfrac{q}4  ( e^{14} + e^{23}). 
$$
This form is not closed, 
$
d \rho_\nabla = \tfrac{q}{2} e^{234} -  \tfrac14  e^{124} 
$.  Now  we use Proposition \ref{riccigeneral} (iii), to compute the first Ricci form of $\nabla^{\Lie{U}(2)}$ which is given by 
$$
 \rho_{\nabla^{\Lie{U}(2)}}  =   \tfrac{1-q^2}8 e^{12} + \tfrac{5(3+q^2)}{8}   e^{34} 
.
$$
It is  in $[\lambda^{1,1}]$ and  its exterior derivative is  $d \rho_{\nabla^{\Lie{U}(2)}}  =  \tfrac{q (1-q^2)}8  e^{234} -  \tfrac{1-q^2}8  e^{124}$. Likewise, by using  Proposition \ref{RicciChern} (ii),  the first Ricci form 
of  the Chern connection is given by 
$$
\rho_{\nabla^h} = - \tfrac{1+q^2}{2} e^{12} +  \tfrac{4+q^2}{2} e^{34} - \tfrac{q}{2}(e^{14} - e^{23}).
$$
It is also in $[\lambda^{1,1}]$ and has non-zero exterior derivative,  $d \rho_{\nabla^h}  = - \tfrac{q(2+q^2)}2
e^{234} + \tfrac{1+q^2}{2}  e^{124} \neq 0 $.
\begin{remark}
{\rm An alternative way to see that our structure is complex is the one indicated in \cite{Martin2} for four-dimensional cases. In such cases one has $\nabla \omega = \xi_+ \otimes \psi_+ + \xi_- \otimes \psi_-$  and the structure is Hermitian ($\mathcal W_4$) if and only if $J \xi_+ = \xi_-$. The advantage is that the one-forms $\xi_+$ and $\xi_-$ can be computed by means of exterior algebra using formulae in the mentioned reference. Thus $2\xi_+= - e_1 - q e_3$ and $2\xi_-= - e_2 - q e_4$. In fact, $J \xi_+ = \xi_-$ in our case.

On other hand, Gauduchon  in \cite{Gau}  claims that each conformal class of Hermitian metrics contains one metric, called {\it standard}, such that the corresponding Lee one-form $\theta$ is coclosed, $d^*\theta =0$. This is the case in the present example.
}
\end{remark}

}
 \end{example}

\begin{example} {\rm 
Let $\mathfrak{g}$ be the Lie algebra with structure equations 
\begin{equation*}
de^i=0, \; \, 1\le i\le4,
\qquad 
de^5=
e^{12},
\qquad
de^6=e^{14}+e^{23}.
\end{equation*} 
This Lie algebra has been included in the list, given by Salamon in \cite{Salamon},  of  real 6-dimensional nilpotent Lie algebras for which the corresponding Lie group $G$ has a left-invariant complex structure.  Because the nilpotent Lie group $G$  has  rational structure constants, there is a discrete subgroup $\Gamma$ such that $M = \Gamma \backslash
 G$ is a compact manifold \cite{Malcev}. 

On $M$ we consider the almost Hermitian structure such that the  metric is the left-invariant on defined by $\langle \cdot , \cdot \rangle = \sum_{i=1}^6 e_i \otimes e_i$ and  its Kähler form is given by 
$$
\omega = e^{6} \wedge e^5  +  (-\tfrac12 e^3 + \tfrac{\sqrt{3}}{2} e^4) \wedge e^1 +   (\tfrac12 e^4 + \tfrac{\sqrt{3}}{2} e^3) \wedge e^2.
$$ 
In \cite{AbbGarSal} it was shown that this structure is Hermitian (see page 162 of \cite{Salamon}, where $\omega = - c(\frac{2 \pi}{3})$ because of our notations). 
The exterior derivative is given by 
$
d \omega = e^{145}  +  e^{235}  - e^{126}.
$
The corresponding Lee form  is 
$
\theta = - \tfrac1{2} J d^*\omega =  \tfrac1{2} \langle \cdot \lrcorner d \omega, \omega \rangle = - \tfrac{\sqrt{3}}{2} e^5
$
Then the  exterior derivative  of $\theta$ is expressed  by 
$
d \theta = - \tfrac{\sqrt{3}}{2} e^{12}.
$
Then  the $\Lie{U}(3)$-components of this  two-form are given by 
$$
(d \theta)_{[\lambda^{1,1}_0]} =  - \tfrac{\sqrt{3}}{4} ( e^{12}  - e^{34}  ), \qquad (d \theta)_{\lcf \lambda^{2,0}\rcf} = - \tfrac{\sqrt{3}}{4} ( e^{12}  + e^{34} ).
$$

In order to see how  the identities of Proposition \ref{divergenciaunouno} work out  in this example, we will  compute the intrinsic torsion $\xi$.  In this case, because $N=0$,  it is given by 
$
4 \xi = (J_{(2)} + J_{(3)}) d \omega.
$
Since components of the exterior derivative $d\omega$ are expressed by 
\begin{eqnarray*}
(d \omega)_{\mathcal{W}_4} & = &  \theta  \wedge \omega = 
   -\tfrac{\sqrt{3}}{4} e^{135} + \tfrac34 e^{145} + \tfrac34 e^{235}  + \tfrac{\sqrt{3}}{4} e^{245},   
\\
(d \omega)_{\mathcal{W}_3} & = & d \omega - \theta   \wedge \omega = 
   - e^{612}  + \tfrac{\sqrt{3}}{4} e^{135} + \tfrac14   e^{145} + \tfrac14 e^{235} - \tfrac{\sqrt{3}}{4} e^{245},   
\end{eqnarray*} 
the corresponding components of $\xi$ are given by 
\begin{eqnarray*}
16 \xi_{(4)} & = & - 2\sqrt{3}  e^1 \otimes e^{15}  - 2\sqrt{3}   e^2 \otimes e^{25} 
 - 2\sqrt{3}    e^3 \otimes e^{36}  - 2\sqrt{3}   e^4 \otimes e^{45} 
\\
&&
- \sqrt{3} e^1 \otimes e^{36}   + 3  e^1 \otimes e^{46}  + 3  e^2 \otimes e^{36} + \sqrt{3} e^2 \otimes e^{46} 
 \\
 &&  
  + \sqrt{3}  e^3 \otimes e^{16}  - 3  e^3 \otimes e^{26}
 - 3  e^4 \otimes e^{16}   - \sqrt{3}  e^4 \otimes e^{26},  
\\
16 \xi_{(3)} & = &
+ 2  e^1 \otimes e^{25} - \sqrt{3}  e^1 \otimes e^{36} -   e^1 \otimes e^{46}
- 2  e^2 \otimes e^{15} + \sqrt{3}  e^2 \otimes e^{46} -   e^2 \otimes e^{36}
\\
&&
- \sqrt{3}  e^3 \otimes e^{16} -   e^3 \otimes e^{26} - 2  e^3 \otimes e^{45}  
-  e^4 \otimes e^{16} + \sqrt{3}  e^4 \otimes e^{26} + 2  e^4 \otimes e^{35}
\\
&&
-4e^5 \otimes e^{12} - 4 e^5 \otimes e^{34} -  2 \sqrt{3} e^6 \otimes e^{13} -   2  e^6 \otimes e^{14}
  - 2 e^6 \otimes e^{23} + 2  \sqrt{3} e^6 \otimes e^{24}.
\end{eqnarray*}
Now doing $\xi = \xi_{(3)} + \xi_{(4)}$, we obtain 
\begin{eqnarray*}
8 \xi & = &  - \sqrt{3}  e^1 \otimes e^{15}  - \sqrt{3}   e^2 \otimes e^{25} 
 - \sqrt{3}    e^3 \otimes e^{35}  - \sqrt{3}   e^4 \otimes e^{45}  
\\
&& 
+  e^1 \otimes e^{25} 
-  \sqrt{3}   e^1 \otimes e^{36}
+  e^1 \otimes e^{46}
-  e^2 \otimes e^{15}  
+  e^2 \otimes e^{36} 
+ \sqrt{3}   e^2 \otimes e^{46}
\\
&&
-2e^3 \otimes e^{26}  
-  e^3 \otimes e^{45}
- 2 e^4\otimes e^{16}
+  e^4 \otimes e^{35}
- 2 e^5 \otimes e^{12} - 2 e^5 \otimes e^{34}
\\
&&
  -   \sqrt{3} e^6 \otimes e^{13} -    e^6 \otimes e^{14}
  -  e^6 \otimes e^{23} +   \sqrt{3} e^6 \otimes e^{24}.
\end{eqnarray*}
From the non-zero Lie brackets: $[e_1, e_2]=-e_5$ and $[e_1 , e_4] =[e_2,e_3]=-e_6$, using Kozsul's formula, the Levi Civita connection is computed and  given by Table \ref{LevCiv}. The minimal connection is now obtained as  $\nabla^{\Lie{U}(3)} = \nabla +\xi$.
\begin{table}[htp]
\begin{center}
\begin{tabular}{ccccccc}
\hline
$\nabla$ & $e_1$ &  $e_2$ & $e_3$ & $e_4$ & $e_5$ & $e_6$  \\[1mm]
\hline
$e_1$ & $0$ &  $-\frac12 e_5$ & $0$ & $-\frac12 e_6$ & $\frac12 e_2$ & $\frac12 e_4$  \\[1mm]
\hline
$e_2$ & $\frac12 e_5$  &  $0$ & $-\frac12 e_6$ & $0$ & $-\frac12 e_1$ & $\frac12 e_3$  \\[1mm]
\hline
$e_3$ & $ 0$  &  $\frac12 e_6$ & $0$ & $0$ & $0$ & $-\frac12 e_2$  \\[1mm]
\hline
$e_4$ & $\frac12 e_6$  &  $0$ & $0$ & $0$ & $0$ & $-\frac12 e_1$  \\[1mm]
\hline
$e_5$  & $\frac12 e_2$ & $-\frac12 e_1$ & $0$ & $0$ & $0$ & $0$  \\[1mm]
\hline
$e_6$ & $\frac12 e_4$   &  $\frac12 e_3$ & $-\frac12 e_2$ & $-\frac12 e_1$ & $0$  & $0$  \\[1mm]
\hline
\end{tabular}\vspace{1mm}
\end{center}
\caption{Levi Civita connection $\nabla$}
\label{LevCiv}
\end{table}

Finally, from all of this, it is straightforward to compute
\begin{gather*}
\textstyle 16 \langle ( \nabla^{\Lie{U}(3)}_{e_i}  \xi_{(3)})_X Y , e_i \rangle  = \left(
  \tfrac{5}{2} \sum_{i=1}^4  e^i \otimes e^{i} 
   - 5 \sum_{i=5}^6  e^i \otimes e^{i} 
   + \tfrac{\sqrt{3}}{2}   (e^{12} - e^{34}) \right)(X,Y),
\\
\langle \xi_{(3)X} Y - \xi_{(3)Y} X ,  \xi_{(4)e_i} e_i  \rangle = - \tfrac{\sqrt{3}}{8} (e^{12} - e^{34})(X,Y),\\
\langle (\nabla^{\Lie{U}(3)}_{e_i} \xi_{(3)})_{e_i} X , Y \rangle =0,
\\
\langle \xi_{(3) \xi_{(4)e_i} e_i} X , Y \rangle = \tfrac{\sqrt{3}}8 (e^{12} + e^{34})(X,Y).
\end{gather*}
Hence one can see in this example how the identities given in Lemma \ref{lambdaunouno} are satisfied.

Now we use the identities \eqref{ricminusricacuno} and \eqref{difsca}  to compute
$$
\textstyle s-s^* = 3, \quad (\Ric - \Ricac)_\mathbb{R}  =\tfrac12 \langle \cdot , \cdot \rangle,  \quad (\Ric - \Ricac)_{[\lambda^{1,1}_0]} = - \tfrac14 \sum_{i=1}^4 e^i \otimes e^i + \tfrac12 \sum_{i=5}^6 e^i \otimes e^i.
$$
Propositions \ref{skewricac} and  \ref{ricsigma} are used to obtain 
$$
\Ricac_{\lcf \lambda^{2,0} \rcf} = (d \theta)_{\lcf \lambda^{2,0} \rcf} =  - \tfrac{\sqrt{3}}{4} ( e^{12}  + e^{34} ),  \qquad \Ric_{\lcf\sigma^{2,0}\rcf} =0.
$$
Next we fix $\Psi = \psi_+ + i \psi_-$, where 
$$
\psi_+   =  e^{125} +  e^{345} - \tfrac12 \left( e^{146} + e^{236}\right) + \tfrac{\sqrt3}{2} \left( e^{246}-e^{136}\right), 
$$
as complex volume form. Since $d\psi_+=0$, then $\eta  = \tfrac{1}{6} \ast(\ast d\psi_+ \wedge \psi_+) + \tfrac13 \theta = - \tfrac{\sqrt3}{6} e^5$. Hence $\hat{\eta} = J \eta =   -\tfrac{\sqrt3}{6} e^6$, $r_{\nabla^{\Lie{U}(3)}} = - 3 d\hat{\eta} = \tfrac{\sqrt3}{2} (e^{14}+ e^{23})$ and $r_{\nabla^h}=0$. Therefore, $r_{\nabla^{\Lie{U}(3)} [\lambda^{1,1}]}(X,JY) = \tfrac34 \sum_{i=1}^4 e^i \otimes e^{i}$. This with 
$
16  \textstyle \langle \xi_{(3)X} e_i , \xi_{(3)Y} e_i \rangle = \sum_{i=1}^4 e^i \otimes e^{i} 
$
and the above considerations complete the ingredients needed to compute, by Proposition \ref{unomastresric}, 
$$
 \textstyle(\Ric + 3 \Ricac)_{\mathbb{R}}  =   -\tfrac{11}{6} \langle \cdot , \cdot \rangle, \qquad 
(\Ric + 3 \Ricac)_{[\lambda^{1,1}_0]}   =   -\tfrac{17}{12} \sum_{i=1}^4 e^i \otimes e^{i}
+ \tfrac{17}{6} \sum_{i=5}^6 e^i \otimes e^{i}. 
$$
Since we have already $s-s^*=3$, it follows
$$
s+3s^* = -11, \qquad s=-\tfrac12, \qquad s^*= - \tfrac72.
$$

From the fact, $\rho_\nabla = r_\nabla = - J_{(2)} \Ricac$, we obtain  
$$
\rho_{\nabla [\lambda^{1,1}]} = \tfrac1{4} e^{65}-\tfrac1{4} \omega, \qquad 
 \rho_{\nabla \lcf \lambda^{2,0}\rcf} = \tfrac{\sqrt{3}}{2} (e^{14} + e^{23}).
$$
Note that $\rho_{\nabla}$ is closed.
Finally, for sake of completeness, we use Proposition \ref{riccigeneral} (iii) to compute the first Ricci form of $\nabla^{\Lie{U}(3)}$ which is given by 
$
 \rho_{\nabla^{\Lie{U}(3)}}    =  e^{65} - \tfrac{3}{4} \omega$.
This Ricci form  is in $[\lambda^{1,1}]$ but it is not closed.  In fact, $ d \rho_{\nabla^{\Lie{U}(3)}} = \tfrac{1}{4} (e^{145} + e^{235} - e^{126}) =   \tfrac{1}{4} d \omega \neq 0$. Likewise, by using  Proposition \ref{RicciChern} (ii),  the first Ricci form 
of the Chern connection is obtained, 
$ \rho_{\nabla^h}  = \tfrac12 e^{65}
-  \tfrac{1}{2} \omega$.  
This form $\rho_{\nabla^h}$ is also in  $[\lambda^{1,1}]$ and  is closed,    $d \rho_{\nabla^h}= 0$.
}
\end{example}

\begin{remark}
{\rm By Proposition \ref{w1w4}, for an almost Hermitian manifold of type $\mathcal{W}_1 \oplus \mathcal{W}_4$ with dimension $2n >4$, the $[\lambda^{1,1}_0]$ component of $d \theta$ vanishes. It would interesting to find an example of the mentioned  type with non-vanishing  $\lcf \lambda^{2,0} \rcf$ component of $d \theta$. Such an example must be of dimension $2n>6$. Note that the type $\mathcal{W}_1 \oplus \mathcal{W}_4$ is specially rigid.	}
\end{remark}

\end{document}